\documentclass[11pt]{article}
\usepackage{amsmath}
\usepackage[margin=2cm]{geometry}
\usepackage{graphicx}
\usepackage{amssymb}
\usepackage{wasysym}
\usepackage{gensymb}
\usepackage{wrapfig}
\usepackage{multicol}
\usepackage{xcolor}
\usepackage{caption}
\usepackage{subcaption}
\usepackage{titlesec}
\usepackage{blindtext}
\usepackage{caption}
\usepackage{enumerate}
\usepackage{amsthm}
\usepackage{url}
\usepackage[pagewise]{lineno}%\linenumbers

\titleformat{\section}
{\normalfont\large\bfseries}{\thesection}{1em}{}
\titleformat{\subsection}
{\normalfont\normalsize\bfseries}{\thesubsection}{1em}{}
\titleformat{\subsubsection}
{\normalfont\small\scshape}{\thesubsubsection}{1em}{}
\newtheorem{theorem}{Theorem}[section]
\newtheorem{prop}[theorem]{Proposition}
\newtheorem{lemma}[theorem]{Lemma}
\theoremstyle{remark}
\newtheorem{remark}{Remark}[section]
\theoremstyle{definition}
\newtheorem{definition}{Definition}[section]
\numberwithin{equation}{section}

\title{\vspace{-1.0cm} Linear Semi-discrete Polyharmonic Flows of Closed Polygons}
\author{\vspace{-2.0cm}James McCoy\footnote{James.McCoy@newcastle.edu.au, ORCID 0000-0001-6053-5144}, Jahne Meyer\footnote{Jahne.Meyer@newcastle.edu.au, ORCID 0009-0001-2850-8848, School of Information and Physical Sciences, University of Newcastle, Australia.  Parts of this work were completed while the first author was supported by Discovery Project DP180100431 of the Australian Research Council and by Chinese Academy of Sciences President's International Fellowship Initiative grant 2024PVA0042.  This work was completed while the second author was supported by an UNRS Central Scholarship and parts of the work was completed whilst supported by a DAAD Research Grant.  The authors are grateful for this support.  MSC 34A26, 51E12  Keywords: system of linear ordinary differential equations, linear curvature flow, evolving polygon}}
\date{}

\begin{document}
\maketitle
\begin{abstract}

In 2007, Chow and Glickenstein considered a linear semi-discrete analogue of the second-order curve shortening flow for smooth closed curves.  In this article, we consider linear semi-discrete analogues of the polyharmonic curve diffusion flows for curves in $\mathbb{R}^p, p\geq 2$.  Since our flows correspond to first-order systems of linear ordinary differential equations with constant coefficients, solutions can be written down explicitly.  As an application of similar ideas, we consider a linear semi-discrete answer to Yau's question of when one can flow one curve to another by a curvature flow.  In this setting, we are able to flow any closed polygonal curve to any other with the same or differing number of vertices, in the sense of exponential convergence in infinite time to a translate of the target polygon.
\end{abstract}

\section{Introduction}

Chow and Glickenstein considered in \cite{chow2007semidiscrete} a semi-discrete, linear analogue of the curve shortening flow, for both planar curves and curves in higher codimension.  Given an ordered collection of $n$ points in $\mathbb{R}^p$, joined in order to form a piecewise linear, closed immersed `curve', the flow is given by a first-order system of $n$ coupled linear ordinary differential equations.  The constant coefficient matrix of this system is a second-order difference type operator.  Later, in \cite{glickenstein2018asymptotic}, Glickenstein and Liang considered a nonlinear generalisation of this flow.  In this article we wish to generalise Chow and Glickenstein's flow in two other directions.  First, we introduce semi-discrete analogues of the polyharmonic curve diffusion flow, whose smooth counterparts were considered in \cite{parkins2016polyharmonic, wheeler2013curve}.  Second, we introduce the semi-discrete analogue of Yau's problem of flowing one smooth curve to another via curvature flow, where in our setting the target curve is given as an ordered collection of $n$ points in $\mathbb{R}^p$, again joined by straight lines to form a piecewise linear closed immersed curve.  In both cases our systems of coupled ordinary differential equations are linear, with the coefficient matrix corresponding to an $m$th order difference operator.  Linearity permits analysis via finite Fourier series.  We find that the higher the order $m$, the faster the convergence, in the first case under appropriate rescaling, to a regular basis polygon, in the second case to the target curve.  This behaviour with respect to $m$ is similarly observed in the smooth case \cite{MSW23}, but there are usually restrictions on the initial data to be able to deduce long time behaviour of the corresponding flows.  We believe the adjustment of Chow and Glickenstein's original semi-discrete flow to the Yau problem is new, as are the `higher order' cases.

Let us mention some further connections between our work and others.  Like the flows of \cite{chow2007semidiscrete} and ours, Vieira and Garcia \cite{VG18} considered semi-discrete flows, continuous in time, but flowing the points by a system of linear ordinary differential equations with arbitrary coefficient matrix.  In a sense, our flows are therefore contained in their class, although the arbitrary coefficient matrix might not in general have an obvious geometric interpretation.  On the other hand, fully discrete flows were considered by Rademacher and Rademacher \cite{rademacher2017solitons, rr21}, work building on foundations that may be traced back to Darboux \cite{darboux1878probleme}. The fully discrete flows are related to discretisation of parabolic partial differential equations in space and time using finite differences, while the semi-discrete flows are related to the `method of lines' where there is discretisation in all but one dimension (usually time), reducing a given partial differential equation to an ordinary differential equation.  The equations we consider are closely related to discretisations via finite differences or finite elements; for related references for the fourth order case we refer the interested reader to \cite{DKS02, CWW23}.  A different kind of fully discrete flow, homotopic curve shortening, is described in \cite{AN21}.

Just as Chow and Glickenstein included in \cite{glickenstein2018asymptotic} a discussion of the relationship between their semi-discrete flow and the classical curve shortening flow of smooth curves, we have here a similar analogy between our flows and the corresponding higher order flows.  In particular, the fourth order \emph{curve diffusion flow} was introduced by Mullins \cite{M57} to model thermal grooving in metals; we refer to  \cite{wheeler2013curve, edwards2016shrinking} and the references contained therein for more about this flow.  While curve shortening corresponds to a heat-type equation for the evolution of the position vector, with the classical Laplacian replaced by the Laplace--Beltrami operator on the evolving curve (hence the nonlinearity), curve diffusion corresponds to a biharmonic heat equation for the position vector, with a twice iterated Laplace--Beltrami operator.  One may continue iterating the Laplace--Beltrami operator, alternating the sign in each case to maintain parabolicity of the resulting flows.  The triharmonic case, with some relevance in computer imaging, was considered in \cite{MPW17}, while the polyharmonic case was considered in \cite{parkins2016polyharmonic}.  These correspond to different powers $m$ of the difference matrix $M$ in our system of ordinary differential equations for vertices evolving under our linear semi-discrete flows.

We remark that settings in which curvature flows result in linear ordinary or partial differential equations are relatively rare.  Apart from the setting of ordinary differential equations in this article, the main other setting where flow equations are linear is where evolving convex curves and hypersurfaces are expressed via the Gauss map parametrisation and flow speeds have a specific structure.  The first of these, studied by Smoczyk \cite{Sm05}, is the expansion of convex curves and hypersurfaces by the reciprocal of the harmonic mean of the principal curvatures.  Recently, the first author, together with Schrader and Wheeler considered higher order linear parabolic curvature flow of strictly locally convex curves \cite{MSW23}.  Notably, as here, some immersed but not embedded curves were allowed, however, unlike here, there were restrictions on the Fourier coefficients of the initial data.  The first author has also recently considered convex hypersurfaces evolving by similar higher order linear parabolic curvature flows \cite{M24}.  Like here, in these cases exponential convergence is faster with increasing order.  He further considered convex hypersurfaces evolving by second order linear hyperbolic flows in \cite{M25}. We remark that the authors have recently considered hyperbolic analogues of the semi-discrete curvature flows of this article in \cite{MM24}.

The structure of this article is as follows.  In Section \ref{S:background}, we set up the necessary notation and background for us to be able to state and prove the main results, including key properties of circulant matrices and the differential operators on which our flows are based.  In Section \ref{sec:prelim}, we introduce the semi-discrete polyharmonic flows and uncover some fundamental properties of these flows.  In Section \ref{S:planar}, we specialise to evolving planar polygons, finding self-similar solutions and a representation formula for general solutions for the flow that allows us to deduce long-time behaviour.  In Section \ref{sec:higher_codimension}, we consider polygons in general codimension.
Finally, in Section \ref{S:Yau}, we consider a nonhomogeneous version of the flow which addresses the discrete analogue of Yau's question of when one can flow one planar curve to another by a curvature flow.

\section{The semi-discrete polyharmonic flow of polygons} \label{S:background}

\begin{definition}\label{def:polygon}

For a fixed integer $p\geq 2,$ we define a closed $n$-sided polygon, or $n$-gon, $X$, as an ordered set of points $X=(X_0, X_1, \ldots, X_{n-1})^T,$ where, for each $j=0,1,\ldots, n-1,$ the vertex $X_j$ is a point in $\mathbb{R}^p$ and $X_n = X_0$.

\begin{remark}
    \begin{enumerate}
        \item It is convenient to think of the index of points of the closed $n$-gon as modulo $n$.
        \item While it is convenient to work with the evolving polygon as defined, it is also useful to think of the polygonal image as each vertex pair $X_i$ and $X_{i+1}$ joined by a line segment.
        \item As for smooth curves, we should distinguish between a polygon as defined, and its image.  Here a given polygonal image might be described by several different `polygons', where the first vertex and direction of vertex numbering may be chosen differently. 
    \end{enumerate}
\end{remark}
    
\end{definition}

 For $j = 0,1,,\ldots, n-1,$ the points $X_j$ are the \textit{vertices} of the polygon,
 and $\overline{X_jX_{j+1}}$ are the line segments that joins each vertex $X_j$ to $X_{j+1}$ for each $j,$ which are considered the \textit{edges}. 
Generally, $X$ is a $n\times p$ matrix with real entries. When $p=2$, the polygon lies in the plane $\mathbb{R}^2.$ In this case, we can alternatively consider each vertex, $X_j = (x_i\  y_i) \in \mathbb{R}^2,$ to be a point in $\mathbb{C}$ such that $ X_j = x_j + iy_j.$ Therefore, $X$ as an $n \times 2$ matrix representing a polygon in $\mathbb{R}^2$ can also be thought of as a vector in $\mathbb{C}^n$.

We can consider a \textit{normal} to each vertex $X_j$ to be given by
\begin{equation} \label{E:Nj}
N_j = (X_{j+1} - X_j) + (X_{j-1} - X_j), 
\end{equation}
such that the corresponding system of equations can be expressed in matrix form as 
\begin{equation}
    N = MX
\end{equation}
where $M$ is an $n \times n$ matrix and is given by
\begin{equation}\label{matrix:M}
M = 
\begin{bmatrix}
-2 & 1 & 0 & \cdots & 0 & 1\\
1 & -2 & 1 & 0 & \cdots & 0 \\
0 & 1 & -2 & 1 & 0 & \vdots\\
\vdots & 0 & \ddots & \ddots & \ddots & 0\\
0 & \ddots & 0 & 1 & -2 & 1\\
1 & 0 & \cdots & 0 & 1 & -2\\
\end{bmatrix}.\\
\end{equation}

\begin{definition}
Let $X(t)$ be a family of polygons as given in Definition \ref{def:polygon}. Given $m\in \mathbb{N},$ polygons $X(t)$ satisfying
\begin{equation}\label{eqn:polyflow}
\frac{dX}{dt} = (-1)^{m+1}M^mX \tag{$\mathrm{SPF}_m$}
\end{equation}
evolve by the $2m$th order \textit{semi-discrete polyharmonic flow}, where $M$ is the matrix given in \eqref{matrix:M}. 
 \end{definition}

\begin{remark}
\begin{enumerate}
  \item When $m = 1,$ \eqref{eqn:polyflow} is the evolution system given by Chow and Glickenstein in \cite{chow2007semidiscrete}. The main results presented there demonstrate that the solutions to the system for $m=1$ shrink to a point as $t$ increases, and asymptotically converge to an affine transformation of a regular polygon. Furthermore, a polygon in higher codimension $\mathbb{R}^p, p\geq 3$ will shrink to a point and asymptotically converge  to a linear transformation of the image of a planar regular convex polygon.  We will prove similar results for the semi-discrete polyharmonic flow, given in Theorems \ref{thm:plane} and \ref{thm:highcodim}. 
  
  \item The elements of $M$ in \eqref{matrix:M} arise as the coefficients in finite difference approximations for the `second derivatives' associated $X$, as in \eqref{E:Nj}.  Powers of $M$ then yield higher finite differences, consistent with those obtained for example by Newton's divided difference approach.  

  \item If we were to consider $m=0,$ the flow \eqref{eqn:polyflow} becomes $\frac{dX}{dt} = - X$ with solution $X(t) = e^{-t}X^0$ for \emph{any} initial $X\left( 0\right) =X^0$.  Therefore the $m=0$ flow shrinks any initial polygon self-similarly and exponentially asymptotically to the origin, with scaling factor $e^{-t}$.  For the remainder of this article we consider $m\in \mathbb{N},$ the strictly positive integers.  We find, for example, only certain polygons can shrink self-similarly under \eqref{eqn:polyflow} for general $m\in \mathbb{N}$, but for other polygons more general behaviour is still completely described by an explicit representation formula for the solution.

  \item As \eqref{eqn:polyflow} is a homogeneous system of linear differential equations with a constant coefficient matrix, existence of a unique solution in a neighbourhood of any initial data is completely standard as one can write down an explicit formula for the solution.  Moreover the formula for the solution reveals that the solution exists for all time, given any initial polygon and indeed \emph{ancient solutions} also make sense for any initial polygon.  Ancient solutions are those that can be extended back $t\rightarrow -\infty$.  These are properties that usually do not hold in full generality for other geometric flows and, when they do, generally require much work to prove.

  \item Let us finally note that uniqueness is understood once the initial polygon has been specified as per Definition \ref{def:polygon}.  
  \end{enumerate}

  \end{remark}

\section{Preliminary properties of the semi-discrete polyharmonic flow}\label{sec:prelim}

In this section we detail the setup of the semi-discrete polyharmonic flow and its properties.

Let $X$ be a $n$-sided polygon as defined in Definition \ref{def:polygon}. We define an ordered set $S,$ of vectors in $\mathbb{R}^p,$ indexed modulo $n$, as 
$$ S = \left\{ V_j \in \mathbb{R}^p\colon j\in \mathbb{Z}/nZ \right\}.$$
We can see that for a polygon $X,$ the set of its vertices are in $S.$

We define a difference operator on these ordered vectors $D\colon S \to S,$ which is the difference of consecutive indexed pairs of vertices of the polygon, that is $D(V_j) = V_{j+1} - V_j$ for all $j = 0, 1, \ldots n-1.$ We denote multiple compositions of $D$ with itself by $D^m$ where $m$ is the number of compositions of $D$. 
It can be shown by induction that for a set of vectors in $S$, namely the vertices on the polygon $X,$ given by $ \{X_0, X_1,\ldots X_{n-1}\},$ we have 
\begin{equation}\label{eqn:Dcomposed}
    D^m(X_j) = \sum^{m}_{k=0}(-1)^{m+k}\binom{m}{k}X_{j+k}
\end{equation}
for all $j = 0,1,\ldots, n-1.$  Also note
$$\frac{dD(X_j)}{dt} = \frac{dX_{j+1}}{dt} - \frac{dX_j}{dt} = D\left(\frac{dX_j}{dt}\right).$$

\begin{definition}\label{def:energyfunctional}
    For any fixed $m\in \mathbb{N}$, define the energy functional $F_m$ on the polygon $X$ by 
\begin{equation}\label{eqn:pfunctional}
    F_m(X) \colon=  \frac{1}{2}\sum^{n-1}_{j = 0} |D^m(X_j)|^2.
\end{equation}
\end{definition}

\begin{lemma}\label{lem:dXj/dt}
 For any fixed $m\in \mathbb{N}$ and polygon $X = (X_0, X_1,\ldots, X_{n-1})^T$, the negative gradient flow for $F_m$ is obtained by taking for each $j=0, \ldots, n-1$
\begin{equation}\label{eqn:dXjdt}
\frac{dX_j}{dt} =  (-1)^{m+1}D^{2m}(X_{j-m}) = (-1)^{m+1}\sum^{2m}_{k=0}(-1)^{k}\binom{2m}{k}X_{j-m + k}.
\end{equation}
\end{lemma}

\begin{proof}
Letting $Y_j = \frac{dX_j}{dt}$, we can compute the variation of the functional for general $m\in \mathbb{N}$ as
\begin{multline*}
    \frac{d}{dt}F_m(X)  
     = \sum^{n-1}_{j =0}\langle D^m(X_j), D^m(Y_j)\rangle
     = \sum^{n-1}_{j =0}\sum^m_{k=0} (-1)^{m+k}\binom{m}{k}\langle D^m(X_j), Y_{j+k}\rangle \\
     = \sum^{n-1}_{j =0}\left\langle \sum^m_{k=0} (-1)^{m+(m-k)}\binom{m}{k}D^m(X_{j-(m-k)}), Y_{j}\right\rangle 
     = -\sum^{n-1}_{j=0}\left\langle(-1)^{m+1}D^{2m}(X_{j-m}), Y_j\right\rangle,
\end{multline*}
from which we conclude the result.
\end{proof}
For $m=1,$ we have 
\begin{equation*}
    \frac{d}{dt}F_1(X) = -\sum_{j=0}^{n-1}\left\langle D^2(X_{j-1}), Y_j\right\rangle = -\sum_{j=0}^{n-1}\left\langle D(X_{j}-X_{j-1}), Y_j\right\rangle = -\sum_{j=0}^{n-1}\left\langle X_{j+1} - 2X_j + X_{j-1}, Y_j\right\rangle,
\end{equation*}
and so the negative gradient flow of $F_1$ is given by
\begin{equation}\label{eqn:pequals1case}
    \frac{dX_j}{dt} = X_{j+1} - 2X_j + X_{j-1},
\end{equation}
for all vertices $X_j$ of polygon $X.$  This is the evolution equation studied in \cite{chow2007semidiscrete}.
Considering the set of equations \eqref{eqn:pequals1case} for all vertices of the polygon, we can write the following equivalent system
\begin{equation}
\frac{dX}{dt} = MX.
\end{equation}
This is the case $m=1$ of \eqref{eqn:polyflow}.

To find the semi-discrete polyharmonic \eqref{eqn:polyflow} for general $m\in \mathbb{N},$ we wish to find an expression for $M^m.$

The matrix $M$ given in \eqref{matrix:M} is a \textit{circulant} matrix. 
A general circulant matrix $B$ is of the form
\begin{equation*}
B = 
\begin{bmatrix}
b_0 & b_1 & b_2 & \cdots & b_{n-2} & b_{n-1}\\
b_{n-1} & b_0 & b_1  & \cdots& b_{n-3} & b_{n-2} \\
b_{n-2} & b_{n-1} & b_0  & \cdots & b_{n-4} & b_{n-3}\\
\vdots & \vdots & \vdots & \cdots & \vdots & \vdots\\
b_2 & b_3 & b_4 &\cdots & b_0 & b_1\\
b_1 & b_2 & b_3 & \cdots  & b_{n-1} & b_0\\
\end{bmatrix}
\end{equation*}
where each row is produced by shifting each of the elements of the previous row to the right. A matrix of this form can also be denoted as $B = \text{circ}(b_0, b_1,\ldots, b_{n-1}).$
Many properties of circulant matrices are detailed in \cite{davismatrices} with further details about their eigenvalue and eigenvector properties in \cite{tee2007eigenvectors}. The product of circulant matrices is also circulant, therefore $M^m$ is circulant for any $m \in \mathbb{N}.$

We will establish an expression for $M^m$ where the terms within the $M^m$ are built from alternating signs of combinatorial terms $\binom{2m}{k}$ for $k=0,1,\ldots, 2m$ as well as zero terms in some cases. As higher powers of $m$ are considered, for $2m\geq n,$ matrix $M^m$ will have some or all entries that are summations of multiple combinatorial terms, and no zero terms. For example, when $n=6$ we have $M = \text{circ}\left(-\binom{2}{1}, \binom{2}{2},0,0,0,\binom{2}{0}\right),$ $M^2 = \text{circ}\left(\binom{4}{2},-\binom{4}{3}, \binom{4}{4}, 0, \binom{4}{0}, -\binom{4}{1}\right),$ $M^3 = \text{circ}\left( -\binom{6}{3}, \binom{6}{4}, -\binom{6}{5}, \binom{6}{6}+\binom{6}{0}, -\binom{6}{1}, \binom{6}{2} \right),$ and so forth where $M^m$ for $m\geq 3$ contains sums of multiple combinatorial terms. To cater for this occurrence we set up a function corresponding to $m$ and $n,$ that takes on values involving the combinatorial terms as well as zero terms. The matrix $M^m$ will be built from a summation expression of these function values. 

\begin{definition}
Given $m\in\mathbb{N},$ $n\in \mathbb{N}, n\geq 3$ and a chosen $r\in \mathbb{N}$ such that $rn - (2m+1) \geq 2,$ we define a function $u_m\colon\mathbb{Z}/nr\mathbb{Z} \to \mathbb{Z}$ by
\begin{equation}\label{eqn:u_function}
u_m(k) \colon= 
\left\{
    \begin{array}{ll}
        (-1)^{m+k}\binom{2m}{m+k}, & \text{for } k = 0,1,\ldots, m\\
        0, & \text{for } k = m+1,m+2,\ldots, rn - m -1\\
        (-1)^{m+k-rn}\binom{2m}{m+k -rn}, & \text{for } k = rn-m,\ldots, rn-1\\ 
    \end{array}
\right. 
\end{equation}
\end{definition}

This function takes on expressions involving combinatorial terms $\binom{2m}{k}$ for $k =0,1\ldots, 2m,$ as well as $rn - (2m+1)$ zero terms. Developing the expression for $M^m,$ given in Lemma \ref{lem:Mpexpression} below and involving sums of $u_m$ function values, requires us to consider $u_{m+1}$ and its relationship to $u_m.$ Therefore $r$ is chosen such that $u_m$ and $u_{m+1}$ have the same domain. Taking different values of $r$ only adjusts the number of zero function values. Provided $rn-(2m+1) \geq 2,$ this ensures $nr \geq 2(m+1)+1$ which is required to encompass all combinatorial terms $\binom{2(m+1)}{k},$ $k=0,1,\ldots, 2(m+1),$ for $u_{m+1}.$ The smallest value $r$ can take on is $\left\lceil \frac{2m+1}{n}\right\rceil$ if $\left\lceil \frac{2m+1}{n}\right\rceil = \left\lceil \frac{2m+3}{n}\right\rceil,$ otherwise $\left\lceil \frac{2m+1}{n}\right\rceil +1.$ 

Given the domain of $u_m,$ we have $u_{m}(k) = u_{m}(nr+k)$ for any $k\in \mathbb{Z}.$ We also note $u_m(k) = u_m(nr-k)$ for any $k\in \mathbb{Z}/nr\mathbb{Z},$ given the relationship between combinatorial terms.

\begin{lemma}\label{lem:ufunction_m+1}

    Take any $m\in \mathbb{N},$ $n\in \mathbb{N},$ where $n\geq 3,$ and choose $r\in \mathbb{N}$ such that $rn - (2m+1) \geq 2.$
    For $u_m$ as defined in \eqref{eqn:u_function}, and for function $u_{m+1}$ given in the same way but with $m+1$ in place of $m,$ we have
    \begin{equation}\label{eqn:ufunction_relationship}
        u_{m+1}(k) = u_{m}(k+1) -2u_{m}(k) + u_{m}(k-1)
    \end{equation}
    for $k = 0, 1, \ldots, rn-1.$ 
     
\end{lemma}
\begin{proof}

We note the following combinatorial expression for $k=1,2,\ldots, 2m-1$:
\begin{multline*}
    (-1)^{k+1}\binom{2m}{k+1} -2(-1)^k\binom{2m}{k} + (-1)^{k-1}\binom{2m}{k-1}\\  = (-1)^{k+1}\left(\binom{2m}{k+1} + 2\binom{2m}{k} +\binom{2m}{k-1}\right)
     = (-1)^{k+1}\binom{2m+2}{k+1}.
\end{multline*}
To establish the relationship between $u_m$ and $u_{m+1}$, we consider different cases of $k.$ 
For $k = 0,1,\ldots, m-1$, we have
\begin{multline*}
u_{m}(k+1) -2u_{m}(k) + u_{m}(k-1) =  (-1)^{m+k+1}\binom{2m}{m+k+1} -2(-1)^{m+k}\binom{2m}{m+k} + (-1)^{m+k-1}\binom{2m}{m+ k-1}\\
 = (-1)^{m+k+1}\binom{2m+2}{m+k+1} =  u_{m+1}(k).
\end{multline*}
For $k=m,$
\begin{multline*}
    u_{m}(m+1) -2u_{m}(m) + u_{m}(m-1)  = -2(-1)^{2m}\binom{2m}{2m} +(-1)^{2m-1}\binom{2m}{2m-1}\\
    = (-1)^{2m+1}\left(\binom{2m+1}{2m+1} + \binom{2m+1}{2m}\right)
    = (-1)^{2m+1}\binom{2m+2}{2m+1} = u_{m+1}(m).
\end{multline*}
For $k=m+1,$
\begin{align*}
   u_{m}(m+2) -2u_{m}(m+1) +  u_{m}(m) = (-1)^{2m}\binom{2m}{2m} = (-1)^{2m+2}\binom{2m+2}{2m+2} = u_{m+1}(m+1).
\end{align*}

Given $u_{m}(k) = u_{m}(nr-k)$ and $u_{m+1}(k) = u_{m+1}(nr-k),$ from the cases above the expression \eqref{eqn:ufunction_relationship} also holds for $k = rn-m-1, rn-m, rn-m+1,\ldots, rn-1.$

For $k = m+2, m+3, \ldots, rn-m-2,$ the expression \eqref{eqn:ufunction_relationship} holds as all $u_{m}(k) = 0$ and $u_{m+1}(k)=0.$
\end{proof}

\begin{lemma}\label{lem:Mpexpression}
    For the $n \times n$ matrix $M$ given by \eqref{matrix:M} and any $m\in \mathbb{N},$ the matrix $M^m$ is a $n\times n$ symmetric circulant matrix given by
    \begin{equation} \label{eqn:Mpexpression}
M^m = \textnormal{circ}\left(b_0,b_1\ldots, b_{n-1}\right),
   \end{equation}
where 
\begin{equation}
    b_k = \sum_{j = 0}^{r-1}u_{m}(jn + k), \mbox{ } k = 0,1,\ldots, n-1,
\end{equation}
and  $u_m$ is as given in \eqref{eqn:u_function} for $r \geq \lceil \frac{2m+3}{n}\rceil.$

\end{lemma}

\begin{proof}
     Since the matrix $M^m$ is circulant by properties of  matrix multiplication, we just need to verify the elements of its first row. We denote rows of $M^m$ as $R_{m,k}$ for $k = 0,1, \ldots, n-1$ with the first row given as $R_{m,0}.$ We follow an induction argument. 
     For the base case when $m=1$, we have $R_{1,0} = (-2,1,0, \ldots, 0,1).$ Since $n\geq 3,$ then $\lceil \frac{2m+1}{n}\rceil = \lceil \frac{3}{n}\rceil = 1.$ We take $r=2$ and so $u_1(0)=-2, u_1(1) = u_1(2n-1)=1$ and $u_1(k) = 0$ for $k=2,3,\ldots, 2n-2.$
     Therefore, $R_{1,0} = (u_{1}(0)+ u_{1}(n), u_{1}(1) + u_{1}(n+1) ,\ldots, u_{1}(n-1)+ u_{1}(2n-1))$ as required. 

    We assume the result is true for $M^m = \text{circ}(b_0,\ldots, b_{n-1})$ such that $b_k = \sum^{r-1}_{j=0}u_m(jn+k)$ for $r \geq \lceil\frac{2m+3}{n}\rceil.$
    
    We calculate $M^{m+1}$ using $u_{m+1}$ and $u_m,$ noting that $u_{m+1}$ and $u_m$ have the same domain and $rn \geq 2(m+1)+1) = 2m+3.$ Denote $M^{m+1}= M^mM$ as $ M^{m+1} = \text{circ}(b'_0,\ldots, b'_{n-1}).$ Given $M$ is symmetric, for each $k=0,\ldots, n-1,$ $b'_k = R_{m,k} \cdot R_{1,0}.$ Therefore, noting $R_{m,k} = (b_{n-k}, b_{n-k+1}, \ldots, b_{n-k-1})$, where the index of the terms of each row of $M^m$ is modulo $n$, and using Lemma \ref{lem:ufunction_m+1}, we have
    \begin{multline*}
        b'_k  = -2b_{n-k} + b_{n-k+1} + b_{n-k-1}\\
          = \sum^{r-1}_{j=0}-2u_{m}(nj+(n-k)) + u_{m}(nj+(n-k+1)) + u_{m}(nj+(n-k-1))
          = \sum^{r-1}_{j=0} u_{m+1}(nj + (n-k))
         \end{multline*}
         and hence
       $$ b'_k = \sum^{r-1}_{j=0} u_{m+1}(nr-(nj+(n-k)))
         = \sum^{r-1}_{j=0} u_{m+1}(n(r-1-j) + k)
     = \sum^{r-1}_{j=0} u_{m+1}(nj +k),$$
 proving the expression for $m+1.$ 

For $k = 0,1,\ldots, n-1,$ we have $u_{m}(k) = u_{m}(nr-k)$ giving
 \begin{equation*}
     b_{n-k} = \sum^{r-1}_{j=0}u_{m}(nj+ n-k) = \sum^{r-1}_{j=0}u_{m}(n(r-1-j)+k) = \sum^{r-1}_{j=0}u_{m}(nj+k)= b_{k},
 \end{equation*}
and so $M^m$ is symmetric for all $m.$
\end{proof}

\begin{lemma}\label{lem:Row_multiple_X_j}
    For $j = 0, 1, \ldots, n-1$, denote by $R_{m,j}$ the $j$th row of the matrix $M^m$ given in Lemma \ref{lem:Mpexpression}.  Then, for a polygon $X = (X_0,\ldots , X_{n-1})^T,$ with $\frac{dX_j}{dt}$ given in \eqref{eqn:dXjdt} for all $j=0,1,\ldots, n-1$, we have
     \begin{equation}
        \frac{dX_j}{dt} = (-1)^{m+1}R_{m,j}X. 
     \end{equation}     
\end{lemma}

\begin{proof}
The $j$th row of $M^m$ is given by $R_{m,j} = (b_{n-j}, b_{n-j+1},\ldots, b_{n-j-1}).$
Therefore
\begin{equation*}
    R_{m,j}X  = \sum_{k=0}^{n-1}b_{k}X_{j+k} = \sum_{k=0}^{r-1}\left(u_{m}(nk)X_j + u_{m}(nk+1)X_{j+1} + \cdots + u_{m}(nk+(n-1))X_{j+(n-1)}\right).
\end{equation*}
Taking into consideration of the zero terms of $u_{m}(k)$ and modulo $n$ indexing of the vertices, we have
\begin{multline*}
    R_{m,j}X = \sum_{k=0}^mu_{m}(k)X_{j+k} + \sum_{k=rn-m}^{rn-1}u_{m}(k)X_{j+k} = \sum_{k=0}^m(-1)^{m+k}\binom{2m}{m+k}X_{j+k} + \sum_{k=0}^{m-1}(-1)^k\binom{2m}{k}X_{j+rn-m+k}\\
    = \sum_{k=m}^{2m}(-1)^k\binom{2m}{k}X_{j-m+k} + \sum_{k=0}^{m-1}(-1)^k\binom{2m}{k}X_{j-m+k} = \sum_{k=0}^{2m}(-1)^k\binom{2m}{k}X_{j-m+k}.
\end{multline*}

Therefore, using Lemma \ref{lem:dXj/dt}, we have for all $j=0,1,\ldots, n-1,$
$$\frac{dX_j}{dt} = (-1)^{m+1}\sum_{k=0}^{2m}(-1)^k\binom{2m}{k}X_{j-m+k} = (-1)^{m+1}R_{m,j}X.$$ 
\end{proof}

From Lemma \ref{lem:Row_multiple_X_j}, we can therefore establish the matrix form of the semi-discrete polyharmonic flow is as given in \eqref{eqn:polyflow} for all $m\in \mathbb{N}.$ 

For any constant $c \in \mathbb{R}$, write $\tilde c = \left( c, \ldots, c \right)^T$ for a corresponding $n$-vector.

\begin{lemma}\label{lem:matrix_properties}
Let $E$ denote a $p \times p$ matrix and let $a_i\in \mathbb{R}$, $i=1, 2, \ldots, p$. We have the following properties:
\begin{enumerate}
    \item $M^m \tilde 1 = \tilde 0.$
    \item $\frac{d}{dt}(XE + (\tilde a_1, \tilde a_2, \ldots ,\tilde a_p)) = (-1)^{m+1}M^m(XE + (\tilde a_1, \tilde a_2, \ldots ,\tilde a_p)).$
\end{enumerate}
\end{lemma}

\begin{proof}
    The sum of each row of $M^m$ is 
    $$ \sum^{2m}_{k=0}(-1)^k\binom{2m}{k} = 0,$$ 
    leading to property 1.  For property 2, using property 1, we have
\begin{equation*}
    \frac{d}{dt}(XE + (\tilde a_1, \tilde a_2, \ldots, \tilde a_p))  = \frac{dX}{dt}E
     = (-1)^{m+1}M^mXE\\
     = (-1)^{m+1}M^m(XE) + (-1)^{m+1}M^m(\tilde a_1, \tilde a_2, \ldots, \tilde a_p)
     \end{equation*}
     and hence
     \begin{equation*}
         \frac{d}{dt}(XE + (\tilde a_1, \tilde a_2, \ldots, \tilde a_p)) 
     = (-1)^{m+1}M^m(XE + (\tilde a_1, \tilde a_2, \ldots, \tilde a_p)).
\end{equation*}
\end{proof}

\begin{remark}
\begin{enumerate}
    \item In \cite{bruckstein1995evolutions}, an analogous flow was considered where time is also discrete. Therein, for any circulant matrix $M,$ the evolution of an initial polygon $P$ at time step $n,$ denoted as $P(n),$ is given by
    \begin{align*}
       P(n) = MP(n-1)\\
       P(0) = P.
    \end{align*}
   Thus $P(n) = M^nP.$

   \item The polygon $X$ can be considered as a graph $G=\{V, E\}$ where $V= \{X_0,\ldots, X_{n-1}\}$ are the vertices of the polygon, and $E$ is the set of edges between consecutive vertices and the degree of each vertex is 2. Therefore, the matrix $L= -M$ is a Laplacian matrix with the corresponding definition as a Laplacian operator; see, e.g., \cite{chung1997spectral}. We have $-L^m = (-1)^{m+1}M^m$ and the semi-discrete polyharmonic flow can be associated with a repeated Laplacian form.
\end{enumerate}
    
\end{remark}

\subsection{Matrix properties}

The eigenvectors of any $n\times n$ circulant matrix are in $\mathbb{C}^n$ and are given by
\begin{equation}\label{eqn:epolygons}P_k = (1, \omega^k, \omega^{2k}, \ldots, \omega^{(n-1)k})^T, k = 0, 1, \ldots , n-1,
\end{equation}
where $\omega = e^{\frac{2\pi i}{n}}.$ 
Here the powers of $\omega$ are precisely the $n$th roots of unity.  We can think of each $P_k$ as a polygon by placing the entries of $P_k$ into $\mathbb{C}$ as the vertices of the polygon, and joining consecutive entries by arrows. This results in $P_0 = (1,1,\ldots,1)^T,$ a point, and the remaining $P_k$ form either regular convex polygons or star shaped regular polygons. Each $P_k$ and $P_{n-k},$ $k=1,\ldots, n-1,$ is a repeat of the regular polygon, but with the arrows in the opposite orientation. The exception is when $n$ is even and then $P_{\frac{n}{2}}$ is a  line interval where the arrows between points overlap each other $\frac{n}{2}$ times. For each $n,$ the collection $\{P_k\}_{k=1}^{n-1}$ produces all regular polygons whose number of sides is a factor of $n$ with $P_{\frac{n}{2}}$ as the exception for even $n.$

\begin{prop}\label{prop:chow}\cite{chow2007semidiscrete}
The set of eigenvectors $\{\frac{1}{\sqrt{n}}P_k\}^{n-1}_{k=0}$ of $M$ forms an orthonormal basis of $\mathbb{C}^n$ where the corresponding eigenvalue $\lambda_k$ for each $P_k$ is given by \begin{equation} \label{E:evalue}
  \lambda_k = -4\sin^2\left(\frac{\pi k}{n}\right).
\end{equation}
\end{prop}

Let the eigenvalues of the matrix $(-1)^{m+1}M^m$ corresponding to the eigenvectors $P_k$ be denoted as $\lambda_{m,k}.$ Then  
\begin{equation}\label{eqn:eigenvalues}
    \lambda_{m,k} \colon= (-1)^{m+1}\lambda_k^m.
\end{equation}
The eigenvalues of a circulant matrix occur in conjugate pairs with some exceptions. That is, $\lambda_k = \overline{\lambda_{n-k}}$ with the exception of $\lambda_0$ and $\lambda_{\frac{n}{2}}$ for $n$ even (\cite{davismatrices, tee2007eigenvectors}). Furthermore, if the circulant matrix is real and symmetric, the eigenvalues are real and thus each eigenvalue pair is $\lambda_k = \lambda_{n-k}$.  This holds for $(-1)^{m+1}M^m$ for each $m$ and so we have $\lambda_{m,0} = 0, \lambda_{m,k} = \lambda_{m,n-k}$ and $\lambda_{m, \frac{n}{2}}$ does not have an equal pair for $n$ even. Also $\lambda_{m,k} < \lambda_{m,1} < 0$ for all $k = 2,\ldots ,\lfloor \frac{n}{2} \rfloor.$

\begin{lemma}\label{lem:nullspace_matrix}
Given a vector $X\in\mathbb{R}^n,$ if $(-1)^{m+1}M^mX = 0$ then $X$ is a constant vector. That is  $X = \tilde c$ for some constant $c\in \mathbb{R}$.
\end{lemma}
\begin{proof}
    The null space of $(-1)^{m+1}M^m$ is given by the eigenspace of $\lambda_{m,0} = 0$ which is spanned by corresponding eigenvector $P_0 = (1,1,\ldots, 1)^T.$ Therefore, any vector in the null space must be a constant vector. 
\end{proof}

\section{Planar solutions to the semi-discrete polyharmonic flow} \label{S:planar}

In this section, we consider solutions of \eqref{eqn:polyflow}.  Firstly, we consider self similar solutions, then, secondly, we consider the solution with general initial data $X(0) = X^0.$

\subsection{Planar self-similar solutions}\label{sec:selfsim_solutions}
Here we are interested in \emph{self-similar solutions} to \eqref{eqn:polyflow}, that is, solutions $X(t)$ related to the initial polygon $X(0) = X^0$ via the formula
\begin{equation}
X(t) =  g(t)X^0R(f(t)) + \mathbf{h}(t),
\end{equation} 
where $g,f \colon\mathbb{R} \rightarrow\mathbb{R}$ are differentiable functions.  Here $g(t)$ represents scaling ($g(t) \neq 0$ for all $t$), $R(f(t))$ is the $2\times 2$ rotation matrix by the angle $f(t),$ and ${h}(t)$ is an $n\times 2$ matrix corresponding to translation, where $\mathbf{h}(t) = (\vec{h}_1(t)\ \vec{h}_2(t))$ and the elements of vectors $\vec{h}_1(t)$ and $\vec{h}_2(t)$, $h_1, h_2\colon \mathbb{R} \rightarrow \mathbb{R},$ are differentiable functions. We have $g(0) = 1,\  f(0) = 0,$ and $\mathbf{h}(0) = 0_{n\times 2},$ where $0_{n\times 2}$ is the $n\times 2$ zero matrix, such that $X(0) = X^0.$ Note that when considering $X$ as a polygon in $\mathbb{C},$ rotation is given by $e^{if(t)}$ and translation $\mathbf{h}(t) \in \mathbb{C}^n.$ 

\begin{prop}\label{prop:selfsimilar}
If a family of polygons in the plane $X(t)$ is a self-similar solution to the flow \eqref{eqn:polyflow} by scaling, then $X(t)$ is of the form 
\begin{equation*}
X(t) = e^{\lambda_{m,k} t}(c_1 P_k + c_2 P_{n-k})
\end{equation*}
for any $k \in \{1, 2, \ldots, \lfloor \frac{n}{2} \rfloor\}$ and any $c_1, c_2 \in \mathbb{C},$
where $\lambda_{m,k}$ is the corresponding eigenvalue for the eigenvectors $P_k$ and $P_{n-k}$ given by \eqref{eqn:eigenvalues}.
\end{prop}

\begin{proof}
    For any fixed $m\in \mathbb{N}$, since $X(t) = g(t)X^0$ is to satisfy equation \eqref{eqn:polyflow} for all $t$, we use this to find an equivalent expression in terms of $X^0.$

We have
\begin{equation}
\frac{dX}{dt} = \frac{dg}{dt}X^0 = (-1)^{m+1}M^mg(t)X^0.
\end{equation}
This can be rearranged to
\begin{equation}\label{eqn:selfsimform}
\frac{g'(t)}{g(t)}X^0 = (-1)^{m+1}M^{m}X^0.
\end{equation}

Since equation \eqref{eqn:selfsimform} holds for all $t$ including $t=0,$ applying the initial conditions, the original polygon $X^0$ satisfies 
\begin{equation}
g'(0)X^0  = (-1)^{m+1}M^mX^0.
\end{equation}

Writing $ a = g'(0)$, the above equation becomes 
\begin{equation}\label{eqn:ABeqn}
a X^0  = (-1)^{m+1}M^mX^0,
\end{equation}
which indicates the function $g$ has to satisfy  $\frac{g'(t)}{g(t)} = g'(0) = a$ for all $t$ for equation \eqref{eqn:selfsimform} to hold. Solving this scalar differential equation with $g\left( 0\right) =1$ gives the function $g(t) = e^{at}.$ 

From \eqref{eqn:ABeqn} we have
\begin{equation*}
    \left[(-1)^{m+1}M^m - a I_n\right]X^0 = 0_{n\times 2},
\end{equation*}
where $I_n$ is the $n\times n$ identity matrix. For this equation to have a nonzero solution $X^0$, we require $\det\left[(-1)^{m+1}M^m - a I_n\right]=0$.  The matrix $(-1)^{m+1}M^m - a I_n$ is circulant, with eigenvalues $\lambda_{m,k} - a$ for $k = 0,1,\ldots \lfloor \frac{n}{2} \rfloor,$ so 
\begin{equation*}
    \det\left[(-1)^{m+1}M^m - a I_n\right] = \prod_{k=0}^{\lfloor \frac{n}{2} \rfloor}(\lambda_{m,k} - a ).
\end{equation*}
Hence, for $(-1)^{m+1}M^m - a I_n$ to be non-invertible, we require $\lambda_{m,k} - a = 0$ for some $k\in\{0,1,\ldots, \lfloor \frac{n}{2} \rfloor\}.$ Since $a$ is constant, it can be equal to any particular $\lambda_{m,k}$; the corresponding null space of $(-1)^{m+1}M^m - \lambda_{m,k} I_n$ is then the span of $P_k$ and $P_{n-k}.$ In the case of $a=\lambda_{m,0}$, we have only the corresponding polygon $P_0$  and similarly if $a=\lambda_{m,\frac{n}{2}}$ and $n$ even, we have only $P_{\frac{n}{2}}$. Therefore, in each case we have the scaling factor $g(t) = e^{\lambda_{m,k}t}$ and solving \eqref{eqn:ABeqn} gives $X^0 = c_1P_k + c_2P_{n-k}$ for complex coefficients $c_1$ and $c_2$ except for the special cases with only one basis element.
When $a = \lambda_{m,0} = 0$, the scaling factor $g \equiv 1$ and so $X(t)$ is just the initial polygon $X^0$ for all $t$ given by $a_0P_0 = \tilde a_0$ in $\mathbb{C}^n$. 
\end{proof}

We can see from Proposition \ref{prop:selfsimilar} that each regular polygon $P_k$ is a self-similar solution, and that scaling self-similar solutions are necessarily affine transformations of these regular polygons.

\begin{prop}\label{prop:selfsim_rotate_planar}
Consider the family of polygons in the plane, $X(t),$ such that \begin{equation}\label{eqn:rotateselfsim}
    X(t) = X^0R(f(t)),
\end{equation}
where $R(f(t))$ represents the rotation of the polygon $X^0$ by the angle $f(t)$ such that $f(0) = 0.$ If $X(t)$ satisfies \eqref{eqn:polyflow} for all $t$, then $f \equiv 0$ and $X^0$ corresponds to a single point in the plane.
That is, there are no nontrivial self-similar solutions that move by pure rotation under the semi-discrete polyharmonic flow.
\end{prop}

\begin{proof}
    Since $X(t) = X^0R(f(t))$ is to satisfy \eqref{eqn:polyflow}, we have
    \begin{equation*}
        X^0\frac{d}{dt} R(f(t)) = (-1)^{m+1}M^mX^0R(f(t)),
    \end{equation*} 
 and so
    \begin{equation*}
        X^0\left[ \frac{d}{dt}R(f(t))\right] R^{-1}(f(t)) = (-1)^{m+1}M^mX^0.
    \end{equation*}
We have 
\begin{equation*}
    \left[ \frac{d}{dt}R(f(t))\right]R^{-1}(f(t))  = -f'(t)R\left(f(t) - \frac{\pi}{2}\right)R(-f(t))\\
     = -f'(t)R\left(-\frac{\pi}{2}\right).
\end{equation*}
Since the above is to be true for all $t,$ then letting $f'(0) = b,$ we also have
\begin{equation}\label{eqn:rotate_expression}
    -bX^0R\left(-\frac{\pi}{2} \right)  = (-1)^{m+1}M^mX^0,
\end{equation}
    and so we require $f'(t) = b$ and thus $f(t) = b\, t + c$ for some constant $c.$

We denote $X^0$ in terms of the coordinates of its vertices, $X^0 = (\vec{x}\ \vec{y})$ where $\vec{x} = (x_0, x_1,\ldots, x_{n-1})^T$ and  $\vec{y} = (y_0, y_1,\ldots, y_{n-1})^T$ and each vertex of $X^0$ is given by $X_j = (x_j, y_j)$ for $j=0,\ldots, n-1.$ Then from \eqref{eqn:rotate_expression} we have
\begin{equation*}
    -b(-\vec{y}\ \vec{x})  = (-1)^{m+1}M^m(\vec{x}\ \vec{y}),
\end{equation*}
and consequently since $-b\vec{x} = (-1)^{m+1}M^m\vec{y}$ and $b\vec{y} = (-1)^{m+1}M^m(-\vec{x}),$ $-b^2\vec{x} = M^{2m}\vec{x}$ and similarly $-b^2\vec{y} =M^{2m}\vec{y}.$
This leads to the expression
\begin{equation}\label{eqn:rotationselfsimconclusion}
    \left[b^2I_n + M^{2m}\right]X^0 = 0_{n\times 2}.
\end{equation}
Again, for a nontrivial solution $X^0$, we require the determinant of $b^2I_n + M^{2m}$ to be equal to zero.
Since $b^2I_n + M^{2m}$ is a circulant matrix, its eigenvalues are $b^2 + \lambda_{m,k}^2$ and thus
\begin{equation*}
    \det\left[b^2I_n + M^{2m}\right] = \prod_{k=0}^{n-1}(b^2 + \lambda_{m,k}^2).
\end{equation*}
There is a zero factor above when $\lambda_{m,0}=0$ provided $b=0$.  Hence $f \equiv c.$ Since $f(0) = 0,$ then $c=0.$  Thus, $X \equiv \tilde a_0$ with complex constant $a_0$, is the only self-similar solution by rotation and is trivial.
\end{proof}

\begin{prop}\label{prop:selfsim_plane_translate}
  Consider the family of polygons in the plane, $X(t),$ such that \begin{equation}\label{eqn:translateselfsim}
    X(t) = X^0 + \mathbf{h}(t),
\end{equation}
where $\mathbf{h}(t)$ represents the translation of the polygon $X^0$ and $\mathbf{h}(0) = \tilde 0.$ If $X(t)$ satisfies \eqref{eqn:polyflow} for all $t$ then $\mathbf{h} \equiv \tilde 0$ and $X^0$ corresponds to a single point in the plane.
That is, there are no nontrivial self-similar solutions by translation under the semi-discrete polyharmonic flow.  
\end{prop}

\begin{proof}
Since $\mathbf{h}(t)$ translates every vertex of $X^0$ in the same way, it follows that $\mathbf{h}(t)$ is a vector in $\mathbb{C}^n$ with every entry equal to the same function of $t.$ Since the function \eqref{eqn:translateselfsim} is to satisfy \eqref{eqn:polyflow}, using properties of the matrix from Lemma \ref{lem:matrix_properties} we have
\begin{equation*}
    \frac{dX}{dt} = \mathbf{h}'(t) = (-1)^{m+1}M^m(X^0 + \mathbf{h}(t))
    = (-1)^{m+1}M^mX^0.
\end{equation*}
Therefore, $\mathbf{h}'(t)$ is a constant vector with all entries equal.  Now the right-hand side above can only be equal to such a vector if every vertex of $X^0$ is the same, that is $X^0 = \tilde a_0$ for a complex constant $a_0.$  It then follows from Lemma \ref{lem:matrix_properties} that $\mathbf{h}'(t) = (-1)^{m+1}M^m \tilde a_0 = \tilde 0$ and so $\mathbf{h} \equiv \tilde c$, for some complex constant $c.$  Since $\mathbf{h}(0)=\tilde 0$ it follows that $\mathbf{h} \equiv \tilde 0$, proving that $X(t) \equiv \tilde a_0$ is the only possible translator and is trivial.
 \end{proof}

\subsection{Planar solutions for general initial data}\label{sec:flow_solutions}

\begin{theorem}\label{thm:plane}
Given an initial polygon $X^0$ with $n$ vertices in $\mathbb{R}^2$ and any $m\in \mathbb{N}$, the equation \eqref{eqn:polyflow} with initial data $X\left(0 \right) = X^0$ has a unique solution given by
\begin{equation} \label{E:flowsoln}
  X(t) = a_0(0)P_0 + \sum_{k=1}^{n-1}a_k(0)e^{\lambda_{m,k}t}P_k.
\end{equation}
The solution exists for all time and converges exponentially to a point.
Under appropriate rescaling, the solution is asymptotic as $t\rightarrow \infty$ to an affine transformation of a regular polygon with $n$ vertices. 

\end{theorem}

\begin{proof}
The proof is similar to the corresponding result in  \cite{chow2007semidiscrete} for $m=1$. 

 The set of eigenvectors $\{P_k\}_{k=0}^{n-1}$ of $(-1)^{m+1}M^m$ forms a basis of $\mathbb{C}^n$. Therefore, by considering $X(t) \in \mathbb{C}^n$ we can write our polygon in the form 
\begin{equation} \label{E:Xform}
  X(t) = \sum_{k=0}^{n-1}a_k(t)P_k,
  \end{equation}
  where the coefficients $a_k(t)$ are complex.  We have 
\begin{align*}
    \frac{dX}{dt} & = \sum_{k=0}^{n-1}\frac{da_k}{dt}P_k,
\end{align*}
and since $X(t)$ satisfies \eqref{eqn:polyflow}, this gives
\begin{equation*}
    \sum_{k=0}^{n-1}\frac{da_k}{dt}P_k  = (-1)^{m+1}M^mX(t)
     = (-1)^{m+1}M^m\sum_{k=0}^{n-1}a_k(t)P_k
     = \sum_{k=0}^{n-1}a_k(t)\lambda_{m,k}P_k.
\end{equation*}

Therefore, 
$$\frac{da_k}{dt} = \lambda_{m,k}a_k(t),$$ with solution
\begin{equation} \label{E:deexp}
  a_k(t) = a_k(0)e^{\lambda_{m,k}t}.
\end{equation}

The value $a_0(0)$ is the (scalar) projection of the initial polygon $X(0)$ onto the eigenvector $P_0.$ That is,
\begin{equation*}
    a_0(0)  = \frac{\left<X(0), P_0\right>}{\left< P_0, P_0\right>}\\
         = \frac{1}{n}\sum_{j=0}^{n-1}X_j(0);
\end{equation*}
this is the centre of mass of the polygon. 

From \eqref{eqn:eigenvalues}, we have $\lambda_{m,0} = 0$. We obtain \eqref{E:flowsoln} by resubstituting the coefficients \eqref{E:deexp} into \eqref{E:Xform}.
    Since all other eigenvalues are strictly negative, 
    $$\lim_{t\to \infty}X(t) = a_0(0)P_0 = (a_0(0), \ldots, a_0(0))^{T}.$$ Thus all the vertices converge to the complex value $a_0(0)$ indicating that the evolving polygon shrinks to a point, specifically its centre of mass. 

To ascertain the limiting shape of the polygon $X(t)$, we consider an appropriate rescaling and translation. Specifically, we scale by factor $e^{-\lambda_{m,1}t}$ and translate by $-\tilde{a}_0(0)$ to consider 
\begin{equation}\label{eqn:rescaledpolygon}
    Y(t) = e^{-\lambda_{m,1}t}(X(t) - \tilde{a}_0(0)).
\end{equation}
We therefore have
\begin{equation*}
Y(t) =  e^{-\lambda_{m,1}t}\left(\sum_{k=0}^{n-1}a_k(0)e^{\lambda_{m,k}t}P_k - \tilde{a}_0(0)\right)\\
     =  a_1(0)P_1 + a_{n-1}(0)P_{n-1} + \sum_{k=2}^{n-2}a_k(0)e^{-(\lambda_{m,1} - \lambda_{m,k})t}P_k.\\
\end{equation*}

The eigenvalue $\lambda_{m,1}$ is chosen within the scaling factor as it is dominant such that $\lambda_{m,1} - \lambda_{m,k} > 0 $ for all $k = 2,\ldots, \lfloor \frac{n}{2} \rfloor,$  and so 
$$\lim_{t\to\infty}Y(t) = a_1(0)P_1 + a_{n-1}(0)P_{n-1}.$$
Therefore $Y$ converges to an affine transformation of $P_1.$ 

If $X(0)$ is orthogonal to $P_1$ and $P_{n-1},$ such that $a_1(0) = a_{n-1}(0) = 0,$ then for any remaining eigenvectors the initial polygon is not orthogonal to, we consider the next dominant eigenvalue $\lambda_{m,k},$ $k\in \{2,\ldots, \lfloor \frac{n}{2} \rfloor\}.$ We follow the same process, taking $Y(t) = e^{-\lambda_{m,k}t}(X - \tilde{a}_0(0))$ such that $\lim_{t\to\infty}Y(t) = a_k(0)P_k + a_{n-k}(0)P_{n-k},$ the affine transformation of $P_k.$
\end{proof}

Figures \ref{fig:general_n=5} and \ref{fig:general_n=6} depict the evolution of a pentagon and hexagon under the semi-discrete polyharmonic flow at a series of time steps for select values of $m$, respectively. They demonstrate the behaviour of the flow as the polygon shrinks and converges to an affine transformation of the regular polygon. The convergence is faster for higher $m$, as expected in view of \eqref{E:flowsoln}. In the case of $n=6$, the dominant eigenvalue is $\lambda_{m,1} = -1$ for all $m$. So while higher values of $m$ do not give faster shrinking to a point, there is faster convergence to the affine transformation of a regular polygon as the terms involving non-dominant eigenvalues tend to zero at a faster rate. This is illustrated in Figure \ref{fig:general_n=6}.

\begin{figure*}

  \begin{subfigure}[b]{0.31\textwidth}
    \includegraphics[width=\textwidth]{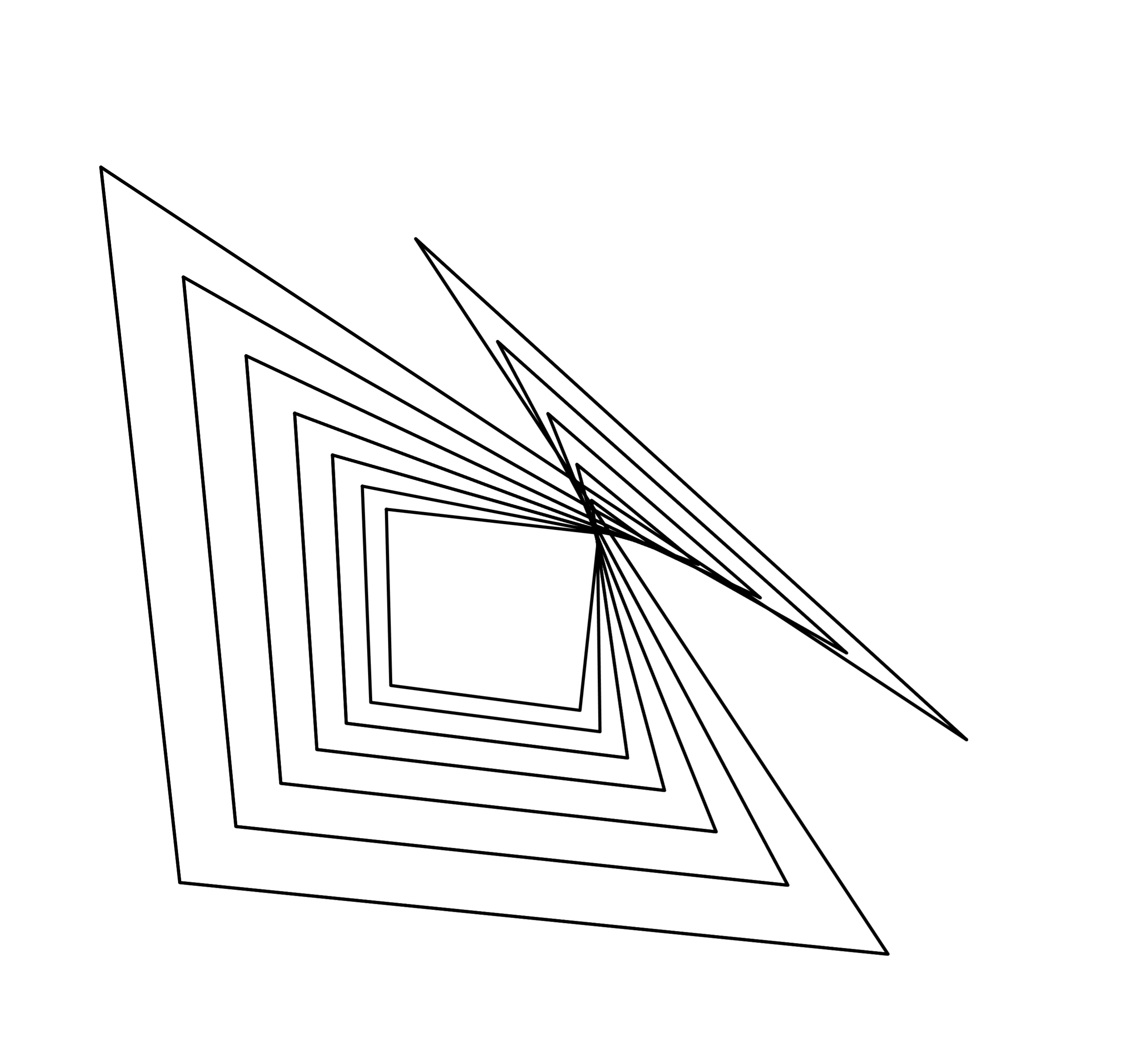}
    \caption{$m=1$}
    \label{fig:general_n=5_m=1}
  \end{subfigure}
  \hfill
  \begin{subfigure}[b]{0.31\textwidth}
    \includegraphics[width=\textwidth]{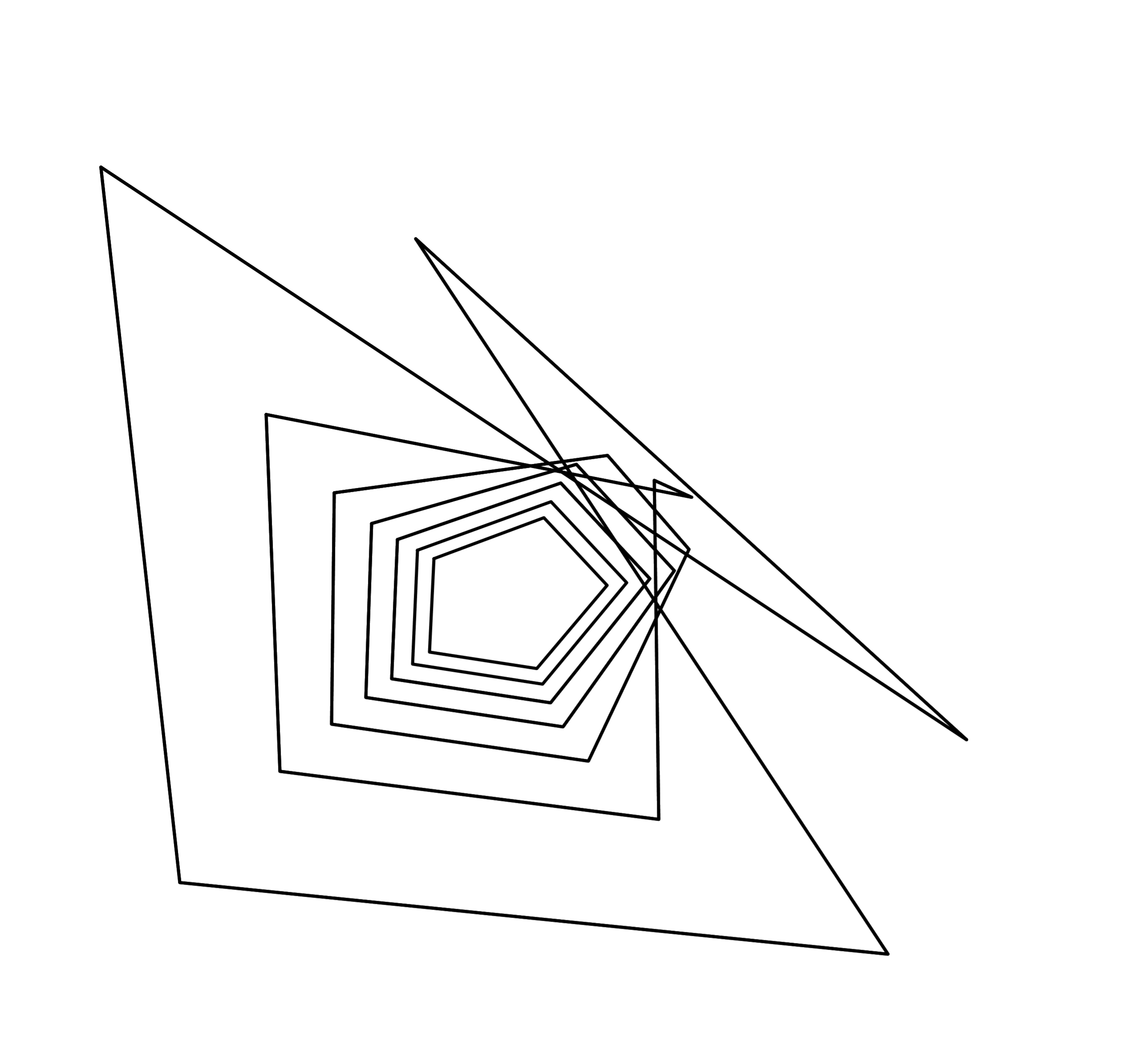}
    \caption{$m=2$}
    \label{fig:general_n=5_m=2}
  \end{subfigure}
  \hfill
   \begin{subfigure}[b]{0.31\textwidth}
    \includegraphics[width=\textwidth]{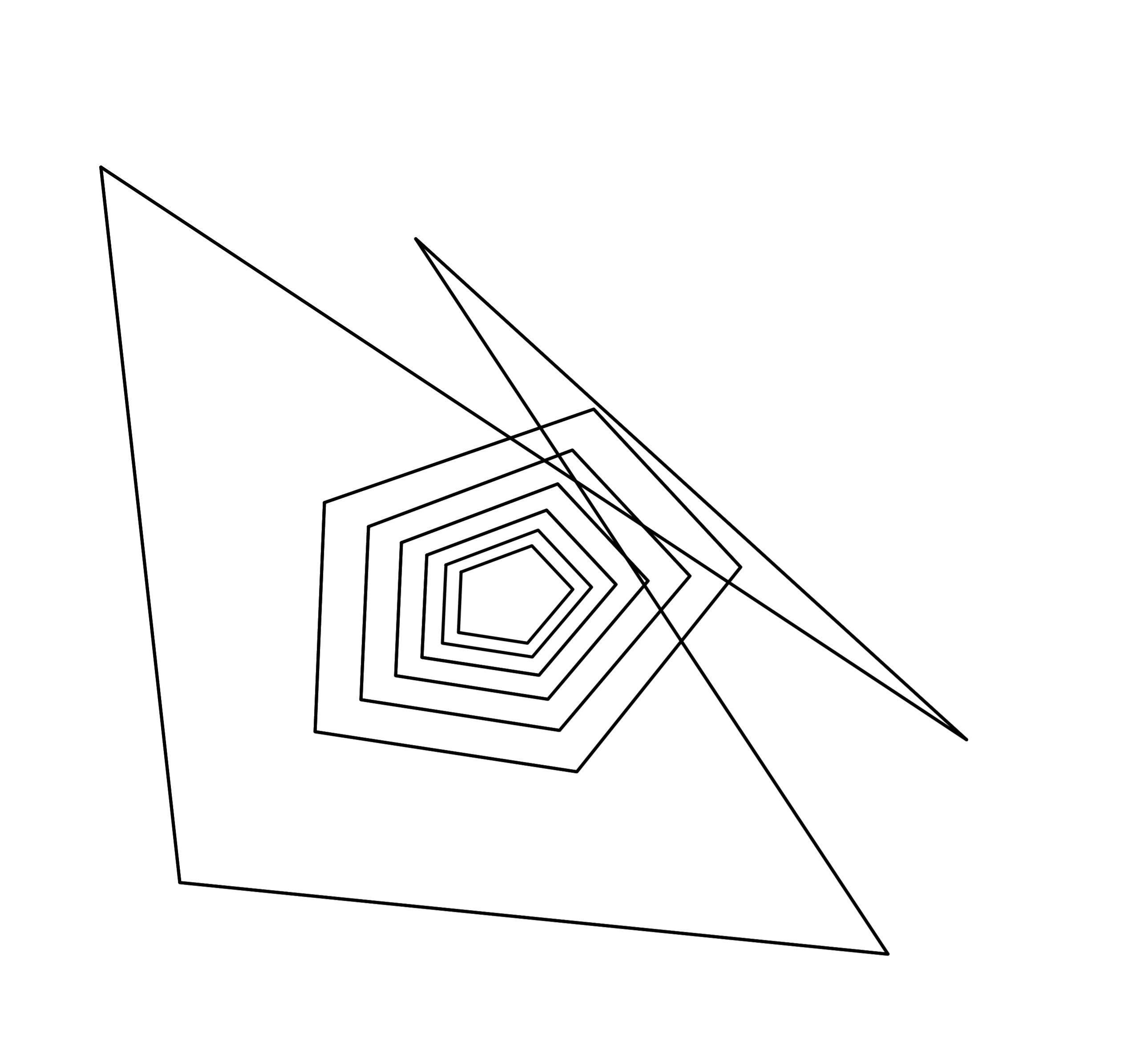}
    \caption{$m=3$}
    \label{fig:general_n=5_m=3}
  \end{subfigure}
 
  \caption{Evolution of a pentagon under the semi-discrete polyharmonic flow for different values of $m.$ Distinct time steps of the evolution are shown superimposed over the initial polygon. The same time step values are used for each case of $m$.} \label{fig:general_n=5}
\end{figure*}

\begin{figure*}

  \begin{subfigure}[b]{0.31\textwidth}
    \includegraphics[width=\textwidth]{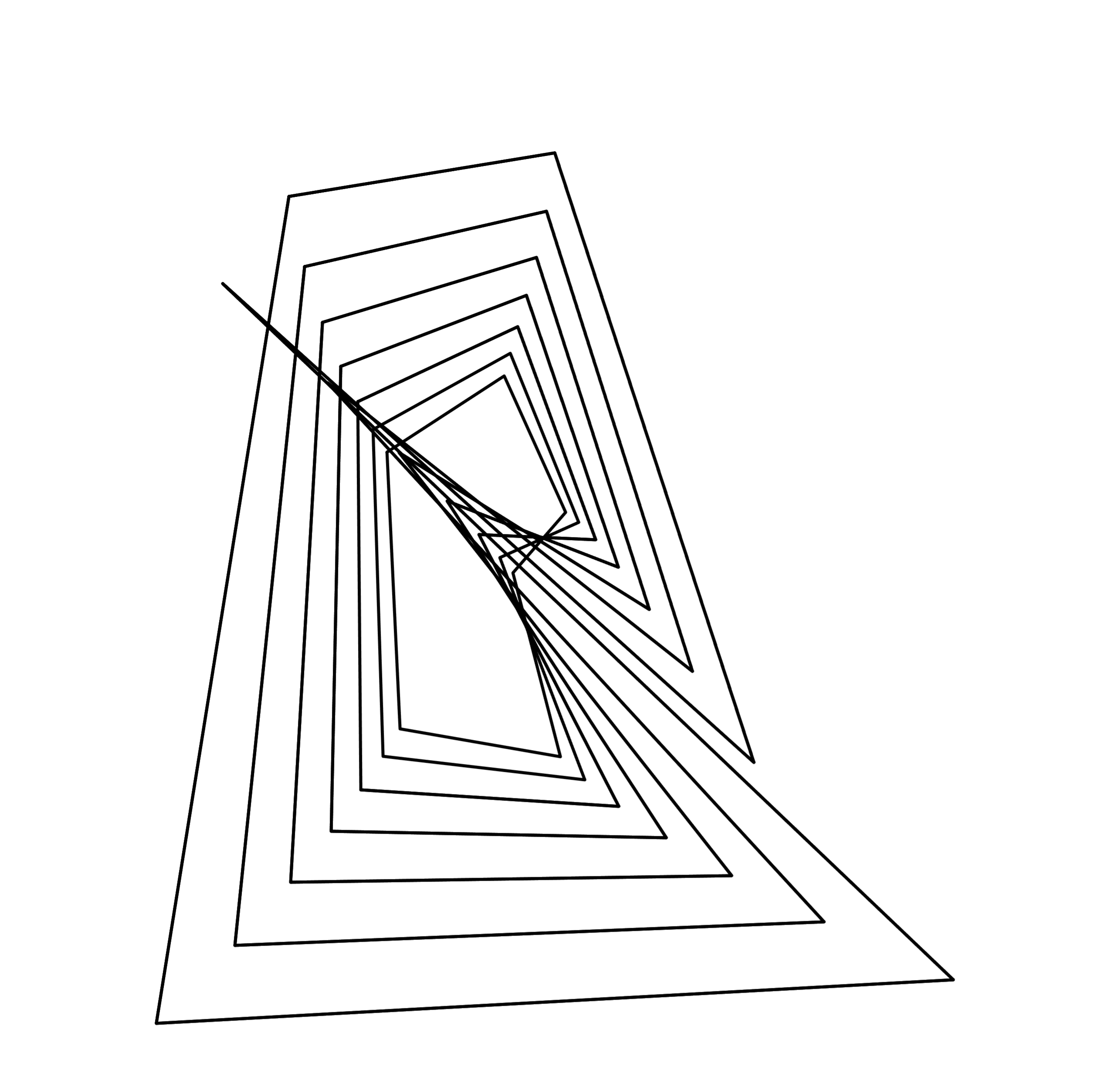}
    \caption{$m=1$}
    \label{fig:general_n=6_m=1}
  \end{subfigure}
  \hfill
  \begin{subfigure}[b]{0.31\textwidth}
    \includegraphics[width=\textwidth]{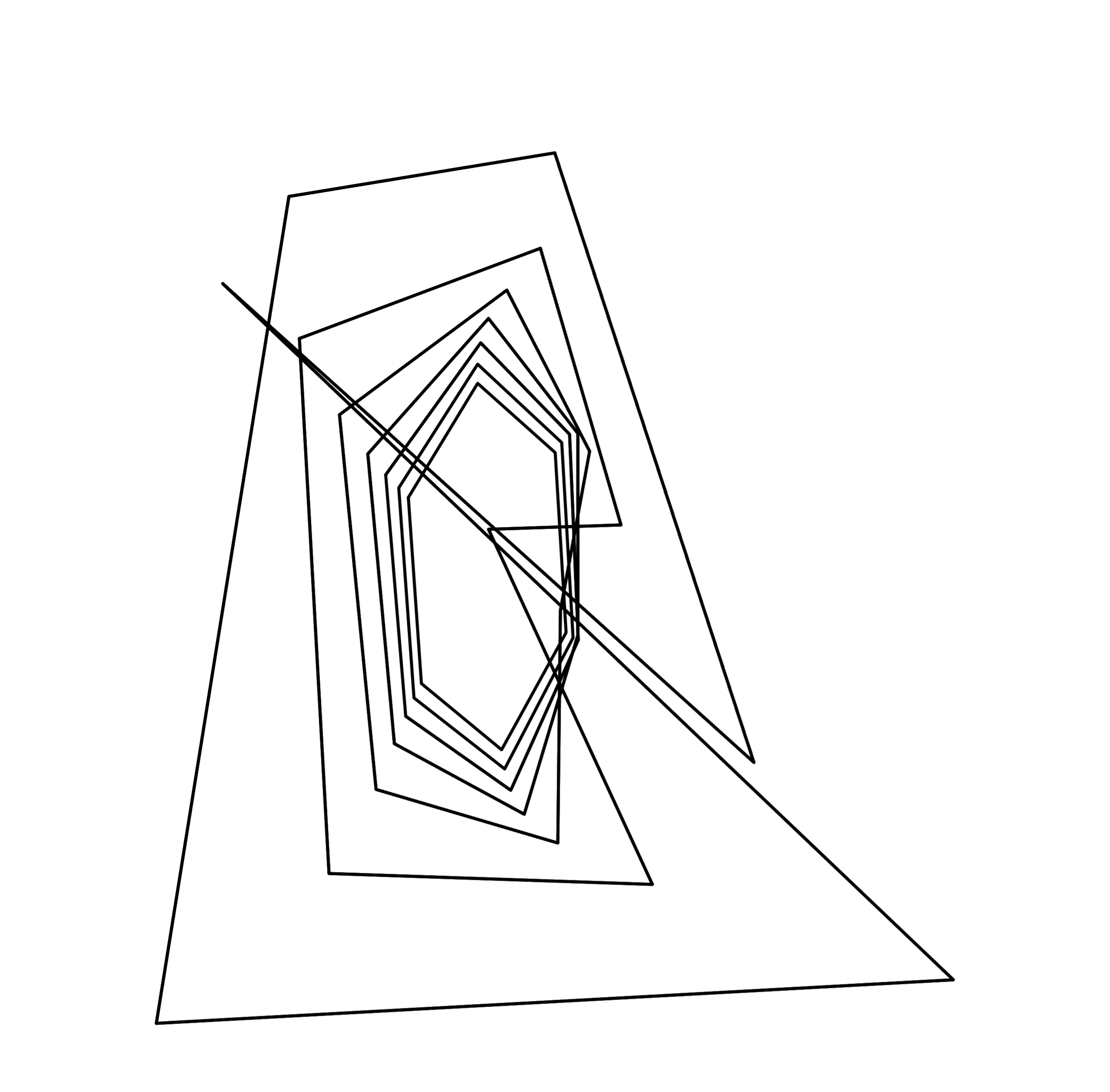}
    \caption{$m=2$}
    \label{fig:general_n=6_m=2}
  \end{subfigure}
  \hfill
   \begin{subfigure}[b]{0.31\textwidth}
    \includegraphics[width=\textwidth]{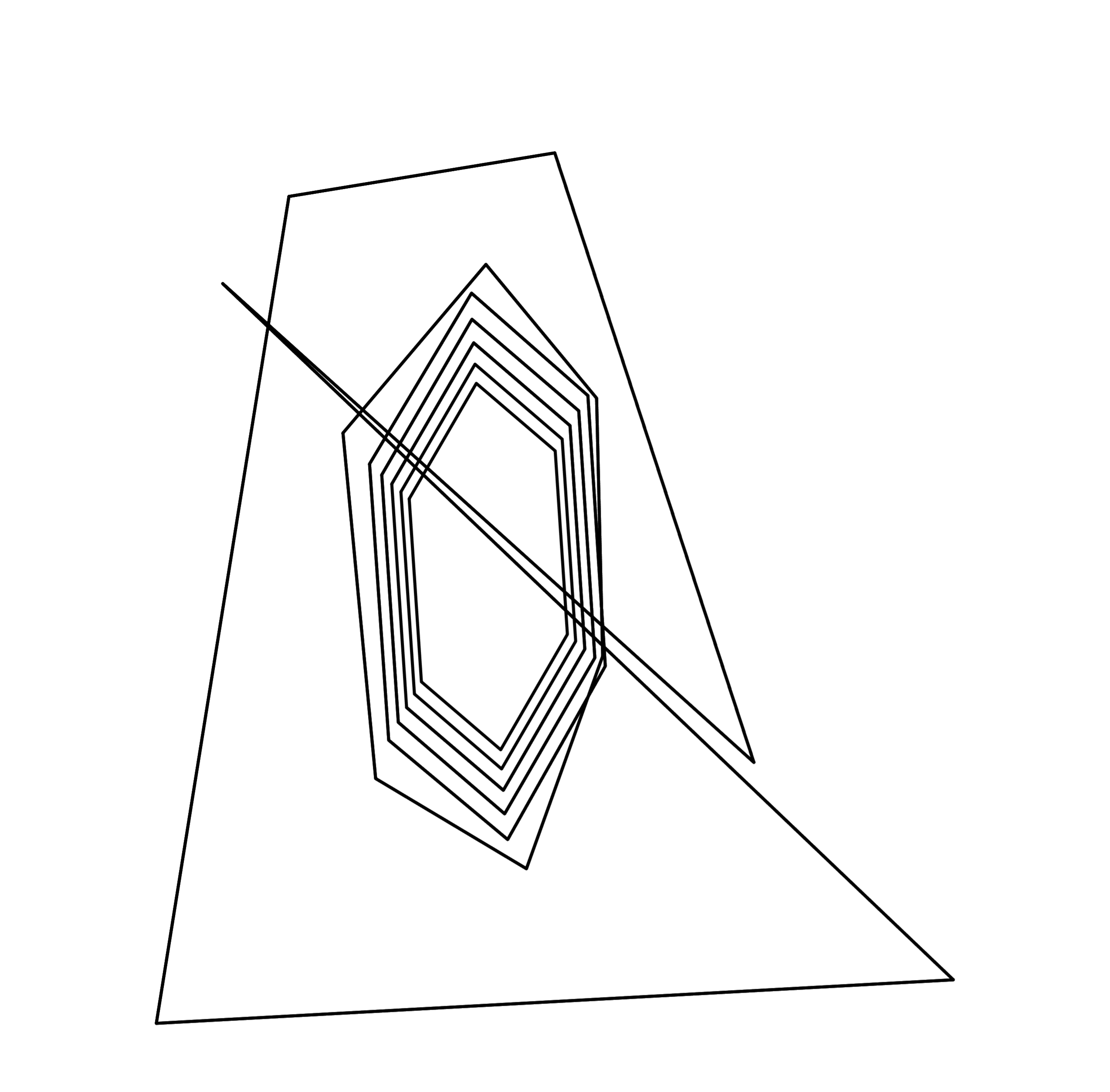}
    \caption{$m=3$}
    \label{fig:general_n=6_m=3}
  \end{subfigure}
 
  \caption{Evolution of a hexagon under the semi-discrete polyharmonic flow for different values of $m.$ Distinct time steps of the evolution are shown superimposed over the initial polygon. The same time step values are used for each case of $m$.} \label{fig:general_n=6}
\end{figure*}

\section{Solutions in higher codimension}\label{sec:higher_codimension}
To consider the flow in higher codimension, we set up similarly as in \cite{chow2007semidiscrete}.  Let each vertex $X_j \in \mathbb{R}^p$ be denoted as $X_j = (x_{1j}, x_{2j},\ldots, x_{pj}).$ Considering the $i$th coordinate for each vertex in the polygon, we define
$$\vec{x}_i = (x_{i0}, x_{i1}, \ldots, x_{i,n-1})^T,$$ which is a vector in $\mathbb{R}^n.$
Therefore,
\begin{equation}\label{eqn:polygon_coordinates}
    X = (\vec{x}_1\  \cdots\  \vec{x}_p).
\end{equation}

For $k = 0, 1, \ldots, \lfloor \frac{n}{2}\rfloor,$ let us define the following vectors in $\mathbb{R}^n$:
\begin{equation}\label{eqn:ck}
c_k  = \left(1, \cos\left(\frac{2\pi k}{n}\right), \cos\left(\frac{4\pi k}{n}\right), \ldots, \cos\left(\frac{2(n-1)\pi k}{n} \right)\right)^T,
\end{equation}
\begin{equation}\label{eqn:sk}
s_k  = \left(0, \sin\left(\frac{2\pi k}{n}\right), \sin\left(\frac{2\pi k}{n}\right), \ldots, \sin\left(\frac{2(n-1)\pi k}{n} \right)\right)^T \mbox{.}
\end{equation}
The vectors $c_k$ and $s_k$ are the real and imaginary parts of the eigenvector $P_k$, respectively. Thinking of each entry of $P_k$ expressed as a vector in $\mathbb{R}^2,$ we have $P_k = (c_k\  s_k).$ Furthermore, nonzero elements from the set $\{c_k, s_k\}_{k=0, 1, \ldots, \lfloor \frac{n}{2} \rfloor}$ are mutually orthogonal and form a basis of $\mathbb{R}^n.$

Therefore, each $\vec{x}_i$ for $i=1,\ldots, p$ can be expressed as 
\begin{equation*}
    \vec{x}_i = \sum_{k=0}^{\lfloor \frac{n}{2} \rfloor}(\alpha_{ik}c_k + \beta_{ik}s_k),
\end{equation*}
where the coefficients $\alpha_{ik}$ and $\beta_{ik}$ are real numbers.

\subsection{Self-similar solutions in higher codimension}

\begin{prop}
   If a family $X(t)$ of polygons with $n$ vertices in $\mathbb{R}^p$ is a self-similar scaling solution to the flow \eqref{eqn:polyflow}, then $X(t)$ has the form 
   \begin{equation}\label{eqn:selfsim_highcodim_scale}
       X(t) = e^{\lambda_{m,k}t}(c_k\ s_k)
    \begin{bmatrix} 
    \alpha_{1} & \cdots & \alpha_{p}\\
    \beta_{1} & \cdots & \beta_{p}\\
    \end{bmatrix}\\
   \end{equation}
   for any particular $k \in \{1, 2, \ldots, \lfloor \frac{n}{2} \rfloor\}$ and real constants $\alpha_j$ and $\beta_j$ for $j = 1,\ldots, p.$  That is, $X(t)$ is the image of a linear transformation of a regular polygon $P_k = (c_k\  s_k)$ with $n$ vertices in $\mathbb{R}^2,$ with linear transformation $T\colon\mathbb{R}^2 \to \mathbb{R}^p$ given by
   \begin{equation}\label{eqn:linear_transform}
T(x,y) = (x \ \ y)\begin{bmatrix} 
    \alpha_1 & \cdots & \alpha_p\\
    \beta_1 & \cdots & \beta_p\\
    \end{bmatrix}.
\end{equation} 
\end{prop}

\begin{proof}
    Suppose $X(t) = g(t)X^0$ for a differentiable scaling function $g\colon \mathbb{R} \rightarrow \mathbb{R}.$ Following the same process as in the proof of Proposition \ref{prop:selfsimilar}, we find $g(t) = e^{at}$ where $a = g'(0).$ Furthermore, we solve 
\begin{equation}\label{eqn:selfsim_highcodim_asolve}
    \left[(-1)^{m+1}M^m - a I_n\right]X^0 = 0_{n\times p}.
\end{equation}

To have $\det\left((-1)^{m+1}M^m - a I_n\right) =0$, we require $a=\lambda_{m,k}$ for some $k \in  \{0,1,\ldots, n-1\}$.  This results in the scaling factor $g(t) = e^{\lambda_{m,k}t}.$  For $a=\lambda_{m,k},$ the null space of $(-1)^{m+1}M^m - a I_n$ is spanned by $P_k$ and $P_{n-k}.$ By Proposition \ref{prop:selfsimilar}, each $P_k$ is self-similar in the plane with scaling factor $e^{\lambda_{m,k}t}.$ This gives $\left((-1)^{m+1}M^m - \lambda_{m,k} I_n\right)P_k = 0_{n\times 2}.$ We write $P_k=(c_k\ s_k)$ and $P_{n-k}=(c_k\ -s_k),$ noting $s_{0} = \tilde{0}$ and $s_{\frac{n}{2}} = \tilde{0}$ for $n$ even.  
We can apply a linear transformation $T$ to each vertex of a polygon in $\mathbb{R}^2,$ as given in \eqref{eqn:linear_transform}, to produce a polygon in $\mathbb{R}^p.$ Therefore
\begin{equation*}
    X^0  = (c_k\ s_k)
            \begin{bmatrix} 
            \alpha_{1} & \cdots & \alpha_{p} \\
            \beta_{1} & \cdots & \beta_{p}\\
            \end{bmatrix}
\end{equation*}
solves \eqref{eqn:selfsim_highcodim_asolve}, where $\alpha_{i}$ and $\beta_{i}$ for all $i = 1,\ldots, p$ are any real coefficients. 
If $a=\lambda_{k,0}=0,$ then $g\equiv 1$ and $X(t) = X^0 = c_0(a_1\ \cdots\  a_p) = (\tilde a_1\ \cdots\ \tilde a_p)$ for real constants $a_1,\ldots, a_p.$ 
\end{proof}

\begin{prop}\label{lem:selfsim_rotate_highcodim}
    Consider the family $X(t)$ of polygons in $\mathbb{R}^p$ such that \begin{equation}\label{eqn:rotateselfsim_highcodim}
    X(t) = X^0R(t),
\end{equation}
where $R\colon \mathbb{R} \rightarrow \mathrm{SO}(p)$ represents the $p\times p$ rotation of the polygon $X^0$ in $\mathbb{R}^p$ such that $R(0) = I_p.$ 

If $X(t)$ satisfies \eqref{eqn:polyflow} for all $t$, then $R \equiv I_p$ and $X^0$ corresponds to a point in $\mathbb{R}^p.$ 
That is, there are no nontrivial self-similar solutions that move by pure by rotation under the semi-discrete polyharmonic flow.
\end{prop}

\begin{proof}
  To assist in the formulation of the rotation matrix as well as its derivative, we also consider the skew-symmetric matrix $S\colon \mathbb{R} \rightarrow \mathfrak{so}(p)$ such that for the exponential map, $\exp\colon\mathfrak{so}(p) \rightarrow \mathrm{SO}(p),$ from the set of skew-symmetric $p\times p$ matrices to the set of $p\times p$ rotation matrices, we have $\exp(S(t)) = R(t).$ Furthermore, the derivative of the rotation matrix is given by
    \begin{equation}\label{eqn:rotation_deriv_highcodim}
        \frac{dR}{dt}= \frac{dS}{dt}R(t).
    \end{equation}
    
Since $X(t) = X^0R(t)$ is to satisfy \eqref{eqn:polyflow}, we have
\begin{equation*}
    X^0\frac{dR}{dt} = (-1)^{m+1}M^mX^0R(t),
\end{equation*}
and from the derivative expression for $R,$ we find
\begin{equation*}
     X^0\left[\frac{dR}{dt}\right]R^{-1}(t) = X^0\frac{dS}{dt} = (-1)^{m+1}M^mX^0.
\end{equation*}

Since the above is to be true for all $t$ then $\frac{dS}{dt}$ is be a constant matrix denoted $B,$ and 
\begin{equation}\label{eqn:rotation_highcodim_scalar}
X^0B  = (-1)^{m+1}M^mX^0.
\end{equation}

To find matrix $B$ and, therefore, the general rotation matrix, we first consider rotations in each $x_ix_j$-plane for $i,j \in \{1,\ldots, p\}$ and $i\neq j,$ where the remaining $p-2$ axes are invariant. We denote such a rotation as $R_{ij}(f_{ij}(t))$ where the differentiable function $f_{ij}\colon\mathbb{R} \rightarrow \mathbb{R}$ is the angle of rotation at time $t$ in the $x_ix_j$-plane and where $f_{ij}(0) = 0.$ Furthermore, $R_{ij}(f_{ij}(t)) = \exp(S_{ij}(f_{ij}(t)))$ for a skew-symmetric matrix $S_{ij}(f_{ij}(t)).$  This matrix $S_{ij}(f_{ij}(t))= [s_{ab}\colon a,b = 1,\ldots, p]$ is given by $s_{ij} = -s_{ji} = -f_{ij}(t)$, and all other the entries are equal to zero. We can also write $S_{ij}(f_{ij}(t)) = f_{ij}(t)S_{ij}(1).$ Denoting $X^0$ in terms of its coordinate vectors $X^0 = (\vec{x}_1\ \cdots\  \vec{x}_p),$ the product $X^0S_{ij}(1)$ swaps the $i$th coordinate vector $\vec{x}_i$ of $X^0$ to $\vec{x}_j,$ and the $j$th coordinate vector $\vec{x}_j$ to $-\vec{x}_i$ and all remaining coordinate vectors become $\tilde{0}.$

The rotation rate of the polygon in the $x_ix_j$-plane is given by $f_{ij}'(t)$ and $\frac{d}{dt}S_{ij}(f_{ij}(t)) = f'_{ij}(t)S_{ij}(1).$ Furthermore, since $\frac{d}{dt}S_{ij}(f_{ij}(t))$ is a constant matrix $B$, it follows that, for some constant $b$, we have $f_{ij}'(t) = b$ for all $t$. 

From \eqref{eqn:rotation_highcodim_scalar} we therefore have
\begin{equation*}
    b\, \vec{x}_j = (-1)^{m+1}M^m\vec{x}_i,\ -b\, \vec{x}_i = (-1)^{m+1}M^m\vec{x}_j \text{ and } \tilde{0} = (-1)^{m+1}M^m\vec{x}_k \text{ for } k \neq i,j.
\end{equation*}
For $\vec{x}_k$ with $k\neq i,j,$ the above indicates $\vec{x}_k = \tilde a_k$  for some real constant $a_k.$ That is, $\vec{x}_k$ is a constant vector, by Lemma \ref{lem:nullspace_matrix}. If we require $\vec{x}_i$ or $\vec{x}_j$ to be nonzero vectors, then a similar argument as in the proof of Proposition \ref{prop:selfsim_rotate_planar} gives $-b^2\vec{x}_i = M^{2m}\vec{x}_i$ and $-b^2\vec{x}_j = M^{2m}\vec{x}_j$ which leads to $b=0$ and consequently $B=0_{p\times p}.$ Therefore, $R_{ij}(t)$ is a constant matrix. Also $\vec{x}_i$ and $\vec{x}_j$ are constant vectors. 

The general rotation $R(t)$ in $\mathbb{R}^p$ is composed of $R_{ij}(f_{ij}(t))$ rotations and is therefore also constant. Given the initial condition, we have $R \equiv I_p.$ If the rotation $R(t)$ acts on a subspace that the polygon is not in, then $X(t) = X^0R(t) = X^0.$
Therefore, the only higher codimension self-similar solutions by rotations in a plane is the trivial solution $X(t) = X^0$ and for this to satisfy the semi-discrete polyharmonic flow, $X^0 = (\tilde{a}_1\  \cdots\  \tilde{a}_p)$ for real constants $a_i$, $i=1,\ldots, p.$

    \end{proof}

\begin{prop}
    Consider the family of polygons with $n$ vertices in $\mathbb{R}^p,$ $X(t),$ such that \begin{equation}\label{eqn:translateselfsim_highcodim}
    X(t) = X^0 + \mathbf{h}(t),
\end{equation}
where $\mathbf{h}(t)$ is a $n\times p$ matrix that represents the translation of the polygon $X^0$ and $\mathbf{h}(0) = 0_{n\times p}.$ If $X(t)$ satisfies \eqref{eqn:polyflow} for all $t$ then $\mathbf{h}(t) = 0_{n\times p}$ for all $t$ and each vertex of $X^0$ is the same fixed point in $\mathbb{R}^p.$
That is, there are no nontrivial self-similar solutions by translation under the semi-discrete polyharmonic flow.  
\end{prop}

\begin{proof}
The proof follows the same process that of Proposition \ref{prop:selfsim_plane_translate}. In the higher codimension case, we find $\mathbf{h}(t) = (\tilde{c}_1\ \cdots\ \tilde{c}_p)$ for real constants $c_i,$  $i = 1,\ldots p.$ However given our initial condition, $c_i=0$ for each $i$.  Therefore $\mathbf{h} \equiv 0_{n\times p}$ and $X(t) = (\tilde{a}_1\ \cdots\  \tilde{a}_p)$ for real constants $a_1,\ldots, a_p$ is the only (trivial) translating self-similar solution. 
      
\end{proof}

\subsection{Higher codimension solutions for general initial data}\label{sec:highcodim_flow_solutions}

A similar result to Theorem \ref{thm:plane} also holds for polygons in higher codimensions. 

\begin{theorem}\label{thm:highcodim}
   Given an initial polygon $X^0$ with $n$ vertices in $\mathbb{R}^p$ and any $m\in \mathbb{N}$, the equation \eqref{eqn:polyflow} with initial data $X\left(0 \right) = X^0$ has a unique solution $X(t)$ for all time that converges exponentially to a point. Under appropriate rescaling, the solution is asymptotic as $t\rightarrow \infty$ to a planar polygon with $n$ vertices in $\mathbb{R}^p$ which is the affine image of a regular polygon in $\mathbb{R}^2.$
\end{theorem}

\begin{proof}
 We argue similarly as the $m=1$ case in \cite{chow2007semidiscrete}. Consider the polygon $X(t)$ expressed in terms of the coordinates of all vertices, that is $X = (\vec{x}_1(t)\  \cdots\  \vec{x}_p(t)).$
Since each $\vec{x}_i(t)$ is in $\mathbb{R}^n,$ we have for each $i=1,\ldots, p,$
\begin{equation*}
    \vec{x}_i(t) = \sum_{k=0}^{\lfloor \frac{n}{2} \rfloor}(\alpha_{ik}(t)c_k + \beta_{ik}(t)s_k)
\end{equation*}
for real coefficients $\alpha_{ik}(t)$ and $\beta_{ik}(t)$  for all $k=0,\ldots, n-1$ and all $t.$ 
Therefore,
\begin{equation*}
    X(t)  = (\vec{x}_1(t)\  \cdots\  \vec{x}_p(t))
             = \sum_{k=0}^{\lfloor \frac{n}{2} \rfloor}(c_k\ s_k)
            \begin{bmatrix} 
            \alpha_{1k}(t) & \cdots & \alpha_{pk}(t)\\
            \beta_{1k}(t) & \cdots & \beta_{pk}(t)\\
            \end{bmatrix}.
\end{equation*}
Considering the derivative and that the polygon is to satisfy \eqref{eqn:polyflow}, we find
\begin{equation*}
    \frac{dX}{dt}  = \frac{d}{dt}(\vec{x}_1(t)\  \cdots\  \vec{x}_p(t))
     =\sum_{k=0}^{\lfloor \frac{n}{2} \rfloor}(c_k\ s_k)
            \begin{bmatrix} 
            \frac{d\alpha_{1k}}{dt} & \cdots & \frac{d \alpha_{pk}}{dt}\\
            \frac{d\beta_{1k}}{dt} & \cdots & \frac{d\beta_{pk}}{dt}\\
            \end{bmatrix},
\end{equation*}
and 
\begin{multline*}
    \frac{dX}{dt}  = (-1)^{m+1}M^m(\vec{x}_1(t)\  \cdots\  \vec{x}_p(t))
         = \sum_{k=0}^{\lfloor \frac{n}{2} \rfloor}(-1)^{m+1}M^m(c_k\ s_k)
            \begin{bmatrix} 
            \alpha_{1k}(t) & \cdots & \alpha_{pk}(t)\\
            \beta_{1k}(t) & \cdots & \beta_{pk}(t)\\
            \end{bmatrix}\\
         = \sum_{k=0}^{\lfloor \frac{n}{2} \rfloor}\lambda_{m,k}(c_k\ s_k)
            \begin{bmatrix} 
            \alpha_{1k}(t) & \cdots & \alpha_{pk}(t)\\
            \beta_{1k}(t) & \cdots & \beta_{pk}(t)\\
            \end{bmatrix}.
\end{multline*}
Therefore $\frac{d\alpha_{ik}}{dt} = \lambda_{m,k}\alpha_{ik}(t)$ and similarly  $\frac{d\beta_{ik}}{dt} = \lambda_{m,k}\beta_{ik}(t),$ leading to the solution
\begin{equation}\label{eqn:high_codim_solution}
    X(t) = \sum_{k=0}^{\lfloor \frac{n}{2} \rfloor}(c_k\ s_k)
    \begin{bmatrix} 
    \alpha_{1k}(0)e^{\lambda_{m,k}t} & \cdots & \alpha_{pk}(0)e^{\lambda_{m,k}t}\\
    \beta_{1k}(0)e^{\lambda_{m,k}t} & \cdots & \beta_{pk}(0)e^{\lambda_{m,k}t}\\
    \end{bmatrix}.
\end{equation}
The coefficients are given by
\begin{align*}
    \alpha_{ik}(0) =\frac{\langle c_k, \vec{x}_i(0)\rangle}{ \langle c_k, c_k \rangle}  \text{ and } \beta_{ik}(0) =\frac{\langle s_k, \vec{x}_i(0)\rangle}{ \langle s_k, s_k \rangle}.
\end{align*}   

To show the long-time behaviour of the polygon under the flow, we follow a similar argument as given in Theorem \ref{thm:plane}, however with our expression for the polygon as given in \eqref{eqn:high_codim_solution}. Given $\lambda_{m,0} = 0$ and all other eigenvalues $\lambda_{m,k}$ are negative for $k = 1, 2, \ldots , \lfloor \frac{n}{2} \rfloor,$ it follows that
\begin{equation*}
    \lim_{t\to \infty}X(t)  = (c_0\ s_0) 
        \begin{bmatrix} 
        \alpha_{10}(0) & \cdots & \alpha_{p0}(0)\\
        \beta_{10}(0) & \cdots & \beta_{p0}(0)\\
         \end{bmatrix}
     = (\tilde{\alpha}_{10}(0)\ \cdots \ \tilde{\alpha}_{p0}(0))
\end{equation*}
and so every vertex of polygon $X(t)$ converges to the same point $ (\alpha_{10}(0), \ldots, \alpha_{p0}(0)) \in \mathbb{R}^p.$

To determine the behaviour of $X(t)$ as it shrinks, we consider a scaled up and shifted polygon, $Y(t),$ given by
\begin{equation*}
    Y(t) = e^{-\lambda_{m,1}t}\left(X(t) - \left(\tilde{\alpha}_{10}(0)\ \cdots \ \tilde{\alpha}_{p0}(0)\right)\right)
\end{equation*}
where $\lambda_{m,1} > \lambda_{m,k}$ for $k = 2, \ldots \lfloor \frac{n}{2} \rfloor.$  We therefore have
\begin{multline*}
    Y(t)  = e^{-\lambda_{m,1}t}\left\{\sum_{k=0}^{\lfloor \frac{n}{2} \rfloor}(c_k\ s_k)
    \begin{bmatrix} 
    \alpha_{1k}(0)e^{\lambda_{m,k}t} & \cdots & \alpha_{pk}(0)e^{\lambda_{m,k}t}\\
    \beta_{1k}(0)e^{\lambda_{m,k}t} & \cdots & \beta_{pk}(0)e^{\lambda_{m,k}t}\\
    \end{bmatrix} 
     - (\tilde{\alpha}_{10}(0)\ \cdots \ \tilde{\alpha}_{p0}(0))\right\}\\
     =  
     (c_1\ s_1)
    \begin{bmatrix} 
    \alpha_{11}(0) & \cdots & \alpha_{p1}(0)\\
    \beta_{11}(0) & \cdots & \beta_{p1}(0)\\
    \end{bmatrix} + \sum_{k=2}^{\lfloor \frac{n}{2} \rfloor}(c_k\ s_k)
    \begin{bmatrix} 
    \alpha_{1k}(0)e^{(\lambda_{m,k}-\lambda_{m,1})t} & \cdots & \alpha_{pk}(0)e^{(\lambda_{m,k}-\lambda_{m,1})t}\\
    \beta_{1k}(0)e^{(\lambda_{m,k}-\lambda_{m,1})t} & \cdots & \beta_{pk}(0)e^{(\lambda_{m,k}-\lambda_{m,1})t}\\
    \end{bmatrix}.\\
\end{multline*}
Therefore 
\begin{equation}\label{highdimlimitshape}
    \lim_{t\to\infty} Y(t)  =(c_1\ s_1)
    \begin{bmatrix} 
    \alpha_{11}(0) & \cdots & \alpha_{p1}(0)\\
    \beta_{11}(0) & \cdots & \beta_{p1}(0)\\
    \end{bmatrix}\\
\end{equation}
We have
\begin{equation*}
    P_1 =  (c_1\ s_1)
\end{equation*}
such that the $j$th vertex of $P_1$ is given as $(\cos(\frac{2\pi j}{n}), \sin(\frac{2\pi j}{n})) \in \mathbb{R}^2.$ Therefore the limit of $Y(t)$ is a polygon in $\mathbb{R}^p$ that has mapped each vertex of $P_1$ by the map $T\colon \mathbb{R}^2 \to \mathbb{R}^p$ given by
\begin{equation*}
T(x,y) = (x \ \ y)\begin{bmatrix} 
    \alpha_{11}(0) & \cdots & \alpha_{p1}(0)\\
    \beta_{11}(0) & \cdots & \beta_{p1}(0)\\
    \end{bmatrix}.
\end{equation*}

This map is linear transformation, and since $P_1$ is in the plane, the image of $P_1$ under $T$ is two dimensional and therefore planar.  

In the case of $\alpha_{i1}(0) =0$ and $\beta_{i1}(0) = 0$ for all $i = 1,2,\ldots p,$ we instead scale the polygon by $e^{-\lambda_{m,k}t}$ such that $\lambda_{m,k}$ is the next dominant eigenvalue where we have non-zero $\alpha_{ik}, \beta_{ik}$, and carry out the same process to show that the polygon converges asymptotically to the image of the linear transformation $T$ of $P_k.$ 
\end{proof}

\begin{remark}
We can examine the limiting shape of ancient solutions here by taking a rescaling and translation of the polygon, similar to the process in Theorem \ref{thm:plane} and \ref{thm:highcodim}, but by the least dominant eigenvalue $\lambda_{m,\lfloor\frac{n}{2}\rfloor},$ and observing behaviour as $t\to -\infty.$ Specifically in the planar case we consider
\begin{equation*}
     Y(t) = e^{\lambda_{m,\lfloor \frac{n}{2} \rfloor}t}(X - \tilde{a}_0(0))
\end{equation*}
such that 
\begin{equation*}
     \lim_{t \to -\infty} Y(t) = a_{\lfloor \frac{n}{2} \rfloor }(0)P_{\lfloor \frac{n}{2} \rfloor} + a_{n-\lfloor \frac{n}{2} \rfloor }(0)P_{n-\lfloor \frac{n}{2} \rfloor},
\end{equation*}
an affine transformation of $P_{\lfloor \frac{n}{2} \rfloor}.$ When $n$ is even, this is a straight line interval of $\frac{n}{2}$ overlapping polygon edges. If the original polygon is orthogonal to $P_{\lfloor \frac{n}{2} \rfloor},$ then the polygon will asymptotically converge to an affine transformation of the regular polygon corresponding to the next least dominant eigenvalue. 
A similar result holds for solutions in higher codimension using the scaled and translated polygon 
\begin{equation*}
    Y(t) = e^{\lambda_{m,\lfloor \frac{n}{2} \rfloor}t}\left(X(t) - \left(\tilde{\alpha}_{10}(0)\ \cdots \ \tilde{\alpha}_{p0}(0)\right)\right)
\end{equation*}
where $\lambda_{m,\lfloor \frac{n}{2} \rfloor} < \lambda_{m,k}$ for $k = 1, 2, \ldots ,\lfloor \frac{n}{2} \rfloor - 1.$ 
Taking $t\to -\infty$ gives a polygon in $\mathbb{R}^p$ that is a linear image of regular polygon $P_{\lfloor\frac{n}{2}\rfloor}$ in $\mathbb{R}^2.$
\end{remark}

\section{Semi-discrete geometric flow between polygons} \label{S:Yau}

We adapt the semi-discrete polyharmonic flow to form an analogue of Yau's curvature difference flow in the semi-discrete setting which we denote as the \textit{semi-discrete Yau difference flow (SYDF)}. Consider $X(t),$ a family of closed polygons in $\mathbb{R}^p$ with $n$ vertices and $Y$ a fixed closed polygon also in $\mathbb{R}^p$ with $n$ vertices.  We set up a \textit{difference polygon}, $Z(t) = X(t) - Y.$
Since $\frac{dZ}{dt} = \frac{dX}{dt}$ and 
\begin{equation*}
    (-1)^{m+1}M^mZ(t) = (-1)^{m+1}M^m(X(t) - Y),
\end{equation*}
applying the semi-discrete polyharmonic flow \eqref{eqn:polyflow} to $Z(t)$ leads to our expression for the semi-discrete Yau difference flow, 
\begin{equation}\label{eqn:Yau_diff_eqn}\tag{$\mathrm{SYDF_m}$}
    \frac{dX}{dt} = (-1)^{m+1}M^m(X(t) - Y).
\end{equation} 
Observe that $X=Y$ is a stationary solution of this flow.

In investigating this difference flow, the diagonalisation of the matrix $(-1)^{m+1}M^m$ is used.  This is given by
\begin{equation}\label{eqn:diagnoalised}
    (-1)^{m+1}M^m = \frac{1}{n}F\text{diag}\left(\lambda_{m,k}\right)\bar{F},
\end{equation}

where $F$ is the Fourier matrix,
\begin{equation}\label{eqn:Fouriermatrix}
    F = \left[P_0 \bigg| P_1 \bigg| \cdots \bigg| P_{n-1}\right] = 
    \begin{bmatrix}
1 & 1 & 1 & \cdots  & 1\\
1 & \omega^1 & \omega^2  & \cdots & \omega^{n-1} \\
1 & \omega^2 & \omega^4 & \cdots  & \omega^{2(n-1)}\\
\vdots & \vdots & \vdots & \ddots & \vdots\\
1 & \omega^{n-1} & \omega^{2(n-1)} & \cdots & \omega^{(n-1)(n-1)}\\
\end{bmatrix},
\end{equation}
and $\text{diag}(\lambda_{m,k})$ is the diagonal matrix with diagonal entries given by the eigenvalues $\lambda_{m,0}, \lambda_{m,1},\ldots, \lambda_{m,n-1}.$ 

The matrix $\bar{F}$ in \eqref{eqn:diagnoalised} is the complex conjugate of the matrix $F$, where $F^{-1} = \frac{1}{n}\bar{F}.$

\begin{theorem}
Let $X^0$ and $Y$ be given $n$-sided polygons.  The flow \eqref{eqn:Yau_diff_eqn}, with $X\left( 0\right) = X^0$,  has a unique solution $X\left( t\right)$ given by

\begin{equation}\label{eqn:Yau_form}
    X(t) = \frac{1}{n}F\textnormal{diag}\left(e^{\lambda_{m,k}t}\right)\bar{F}\left(X^0 - Y\right) + Y,
\end{equation}
where $F$ is given by \eqref{eqn:Fouriermatrix}.  In particular, the solution exists for all time and as $t\rightarrow \infty$ converges exponentially to a translate of $Y$.
\end{theorem}

\begin{proof}
Since $X(t)$ and $Y$ are closed polygons in $\mathbb{R}^p$ they can be expressed as
\begin{equation*}
    X(t)  = (\vec{x}_1(t)\  \cdots\  \vec{x}_p(t))
             = \sum_{k=0}^{\lfloor \frac{n}{2} \rfloor}(c_k\ s_k)
            \begin{bmatrix} 
            \alpha_{1k}(t) & \cdots & \alpha_{pk}(t)\\
            \beta_{1k}(t) & \cdots & \beta_{pk}(t)\\
            \end{bmatrix}
\end{equation*}
and 
\begin{equation*}
    Y  = (\vec{y}_1\  \cdots\  \vec{y}_p)
             = \sum_{k=0}^{\lfloor \frac{n}{2} \rfloor}(c_k\ s_k)
            \begin{bmatrix} 
            \gamma_{1k} & \cdots & \gamma_{pk}\\
            \delta_{1k} & \cdots & \delta_{pk}\\
            \end{bmatrix},
\end{equation*}

where $c_k$ and $s_k$ are as given in \eqref{eqn:ck} and \eqref{eqn:sk}, and $\alpha_{ik}(t), \beta_{ik}(t), \gamma_{ik}, \delta_{ik}$ are real coefficients for all $i=1,\ldots, p$ and $k = 0,1,\ldots,n-1.$  The vectors $\vec{x}_i(t)$ and $\vec{y}_i$ in $\mathbb{R}^n$ represent the $i$th coordinates of all vertices in $X(t)$ and $Y$, respectively. 

The equation given in \eqref{eqn:Yau_diff_eqn} is a inhomogeneous system of ordinary differential equations.
Using the diagonalisation of the coefficient matrix, given in \eqref{eqn:diagnoalised}, the complementary solution to \eqref{eqn:Yau_diff_eqn}, that is, the solution to the associated homogeneous equation is 
\begin{equation}\label{eqn:homogenous_solution}
    X_c(t) = \frac{1} {n}F\text{diag}\left(e^{\lambda_{m,k}t}\right)\bar{F}X^0.
\end{equation}

For the particular solution of \eqref{eqn:Yau_diff_eqn}, we compute
\begin{multline*}
    X_p(t) 
     = e^{(-1)^{m+1}M^m t}\int^t_0 e^{-(-1)^{m+1}M^m s}(-1)^{m}M^mY ds\\ 
     = \frac{1} {n}F\text{diag}\left(e^{\lambda_{m,k}t}\right)\bar{F}\frac{1} {n}F\left[\int^t_0 \text{diag}\left(e^{-\lambda_{m,k}s}\right)ds \right]\bar{F}(-1)^{m}M^mY
     = Y -  \frac{1}{n}F\text{diag}\left(e^{\lambda_{m,k}t}\right)\bar{F}Y.
\end{multline*}
Therefore, 
\begin{align*}
    X(t) 
    & = \frac{1} {n}F\text{diag}\left(e^{\lambda_{m,k}t}\right)\bar{F}X^0 + Y -  \frac{1}{n}F\text{diag}\left(e^{\lambda_{m,k}t}\right)\bar{F}Y,\\
\end{align*}
which leads to \eqref{eqn:Yau_form} with a rearrangement.  We have $X(0) = \frac{1}{n}F\bar{F}\left(X^0 - Y\right) + Y = X^0.$

To show that the polygon converges to a translation of the fixed polygon $Y$, we note that $c_0 = (1,1,\ldots, 1)^T$ and
\begin{equation*}
    \alpha_{i0}(0) =\frac{\langle c_0, \vec{x}_i(0)\rangle}{ \langle c_0, c_0 \rangle} = \frac{1}{n}\langle c_0, \vec{x}_i(0)\rangle
    \text{ and } \gamma_{i0} =\frac{\langle c_0, \vec{y}_i\rangle}{ \langle c_0, c_0 \rangle} = \frac{1}{n}\langle c_0, \vec{y}_i\rangle.
\end{equation*}     
Therefore, since $\lambda_{m,0}=0$ and $\lambda_{m,k}<0$ for $k=1,2,\ldots,n-1,$
\begin{multline*}
    \lim_{t\to \infty}X(t) 
     = \lim_{t\to \infty}\frac{1}{n}F\text{diag}\left(e^{\lambda_{m,k}t}\right)\bar{F}\left(X^0 - Y\right) + Y
     = \frac{1}{n}F\text{diag}(1,0,\ldots,0)\bar{F}\left(X^0-Y\right) + Y\\
     = \frac{1}{n}(c_0\  c_0 \cdots c_0)\left((\vec{x}_1(0)\ \cdots\ \vec{x}_p(0)) - (\vec{y}_1\  \cdots\  \vec{y}_p)\right)+ Y
     = \left(\tilde{\alpha}_{10}(0)\ \cdots\ \tilde{\alpha}_{p0}(0) \right) - \left(\tilde{\gamma}_{10}\ \cdots\ \tilde{\gamma}_{p0} \right) + Y,
\end{multline*}
which is the fixed target polygon $Y$ translated to the centre of mass of the difference polygon $X^0 - Y.$
\end{proof}

Figures \ref{fig:Yau_cases} and \ref{fig:Yau_vertex_cases} depict some examples of the semi-discrete Yau difference flow. The initial and target polygons can have differing numbers of vertices.  The polygon with the lower number of vertices simply has enough vertices duplicated (Figures \ref{fig:Yau_multi_m=1} and \ref{fig:Yau_multi_m=3}), or vertices added along existing connecting lines (Figures \ref{fig:Yau_4_m=2} and \ref{fig:Yau_mid_m=3}), to provide the same number of points in the initial data and target polygon.  Either or both initial and target polygon could be a (multiply-covered) point or line segment.\\

\begin{figure*}

  \begin{subfigure}[b]{0.31\textwidth}
    \includegraphics[width=\textwidth]{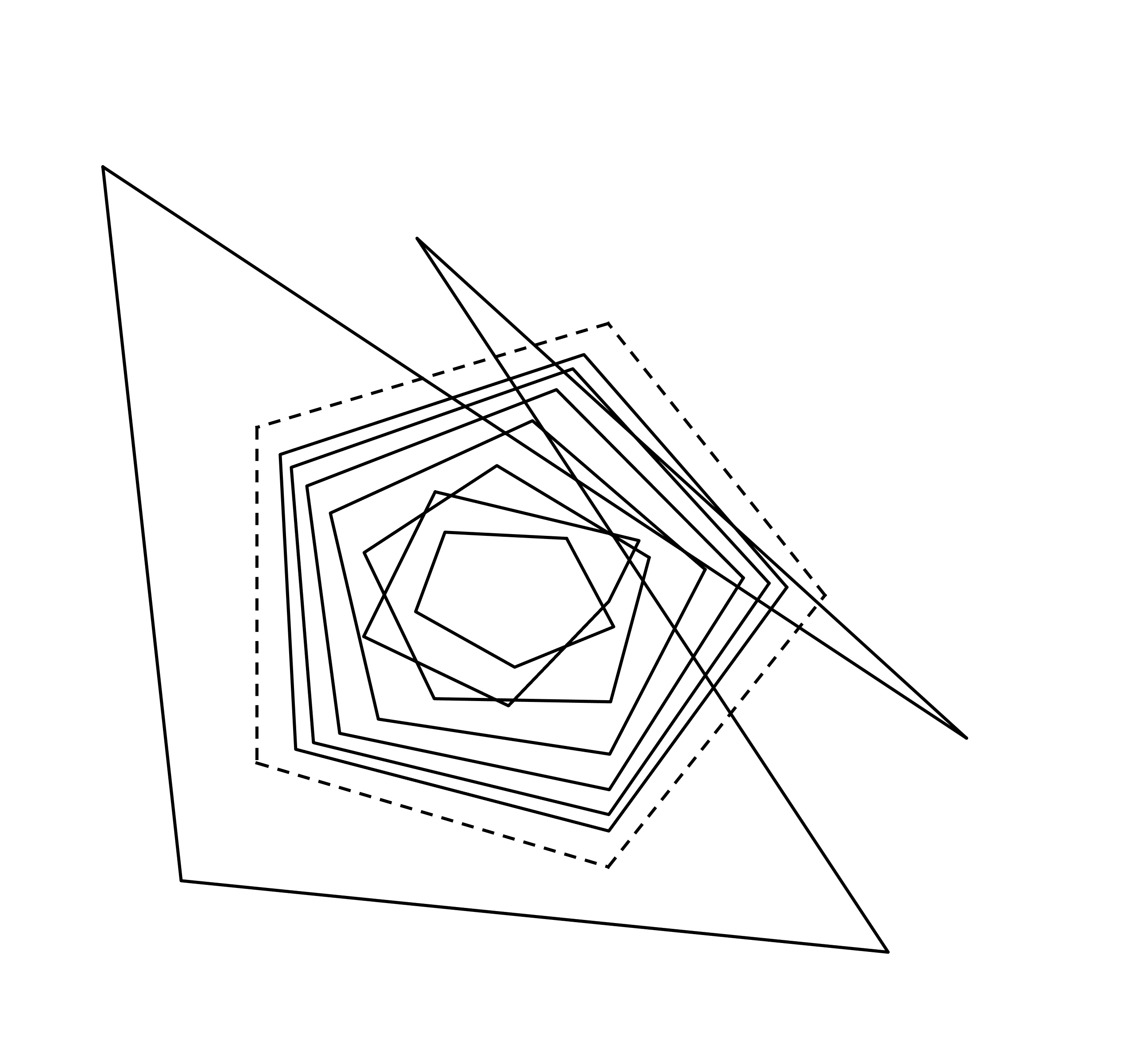}
    \caption{Pentagon flowing to regular convex pentagon by \eqref{eqn:Yau_diff_eqn} for $m=2.$}
    \label{fig:Yau_1_m=2}
  \end{subfigure}
  \hfill
  \begin{subfigure}[b]{0.31\textwidth}
    \includegraphics[width=\textwidth]{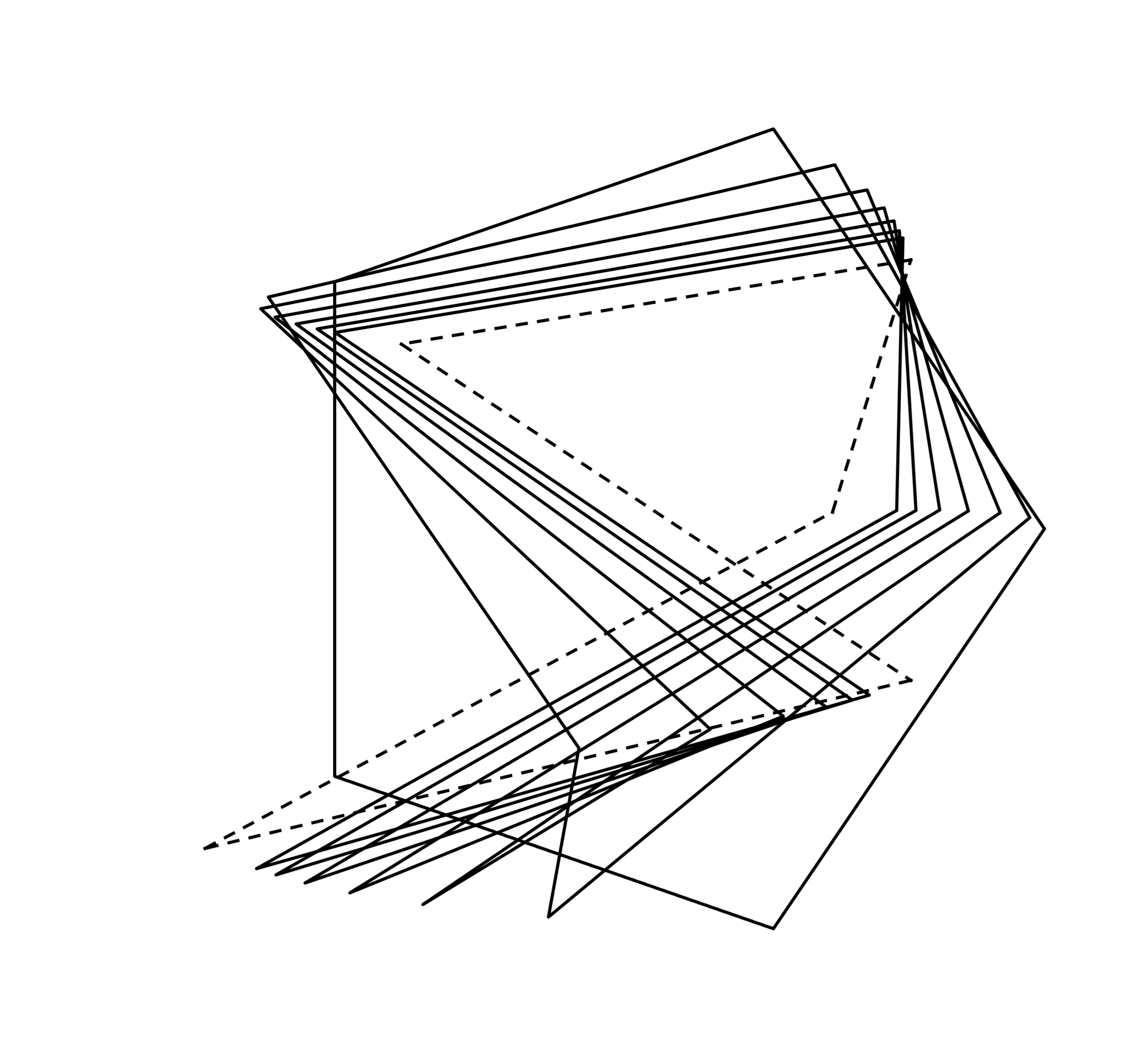}
    \caption{Regular convex pentagon flowing to an irregular pentagon by \eqref{eqn:Yau_diff_eqn} for $m=1.$}
    \label{fig:Yau_2_m=1}
  \end{subfigure}
  \hfill
   \begin{subfigure}[b]{0.31\textwidth}
    \includegraphics[width=\textwidth]{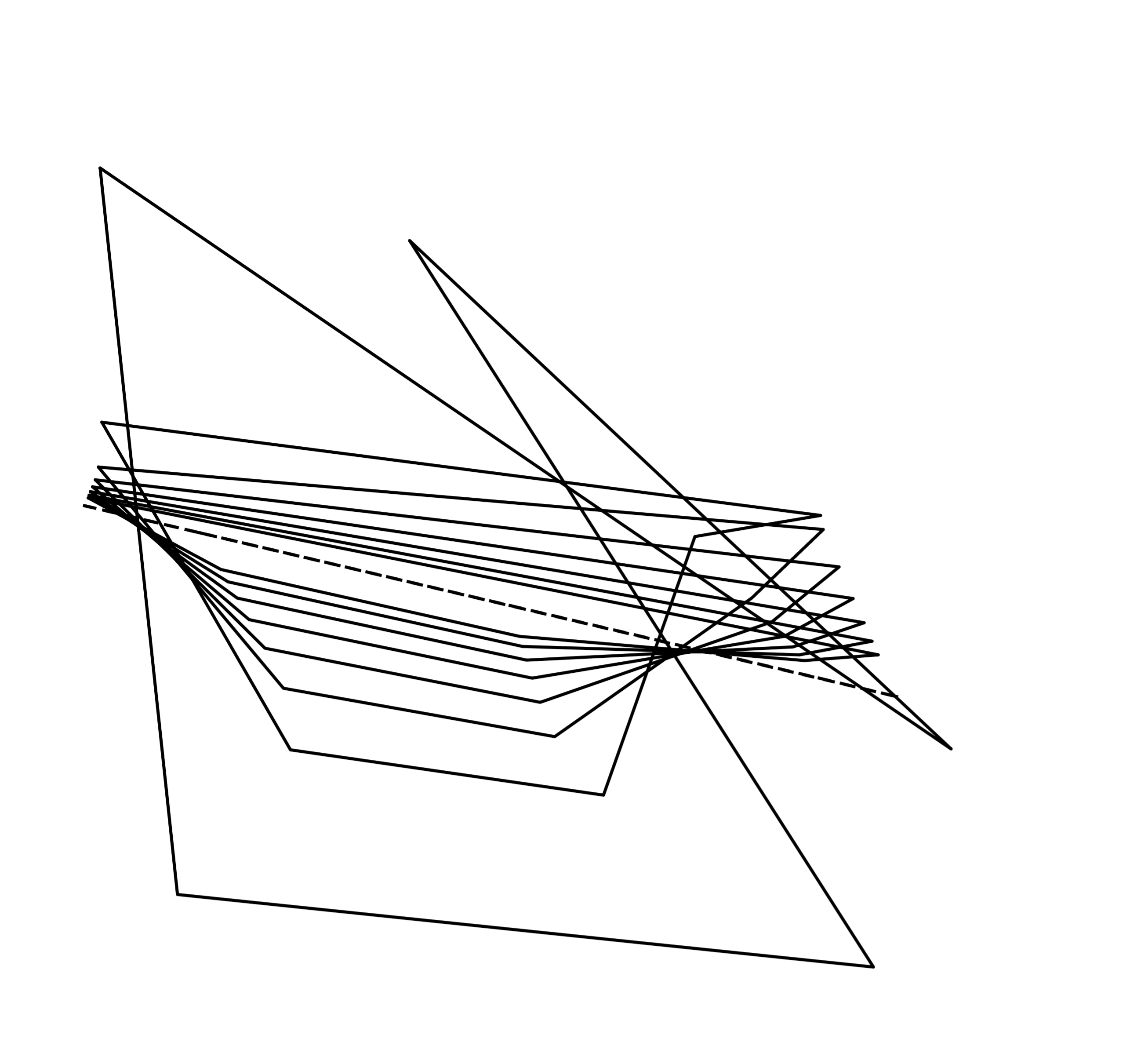}
    \caption{Irregular pentagon flowing by \eqref{eqn:Yau_diff_eqn} for $m=2$ to a line segment.}
    \label{fig:Yau_4_m=2}
  \end{subfigure}
 
  \caption{Different cases of pentagons flowing to other pentagons under the semi-discrete Yau difference flow. In each case, selected time steps of the evolution are shown superimposed over the initial and target polygons. The target polygon is given by the dashed line.} \label{fig:Yau_cases}
\end{figure*}

\begin{figure*}

  \begin{subfigure}[b]{0.31\textwidth}
    \includegraphics[width=\textwidth]{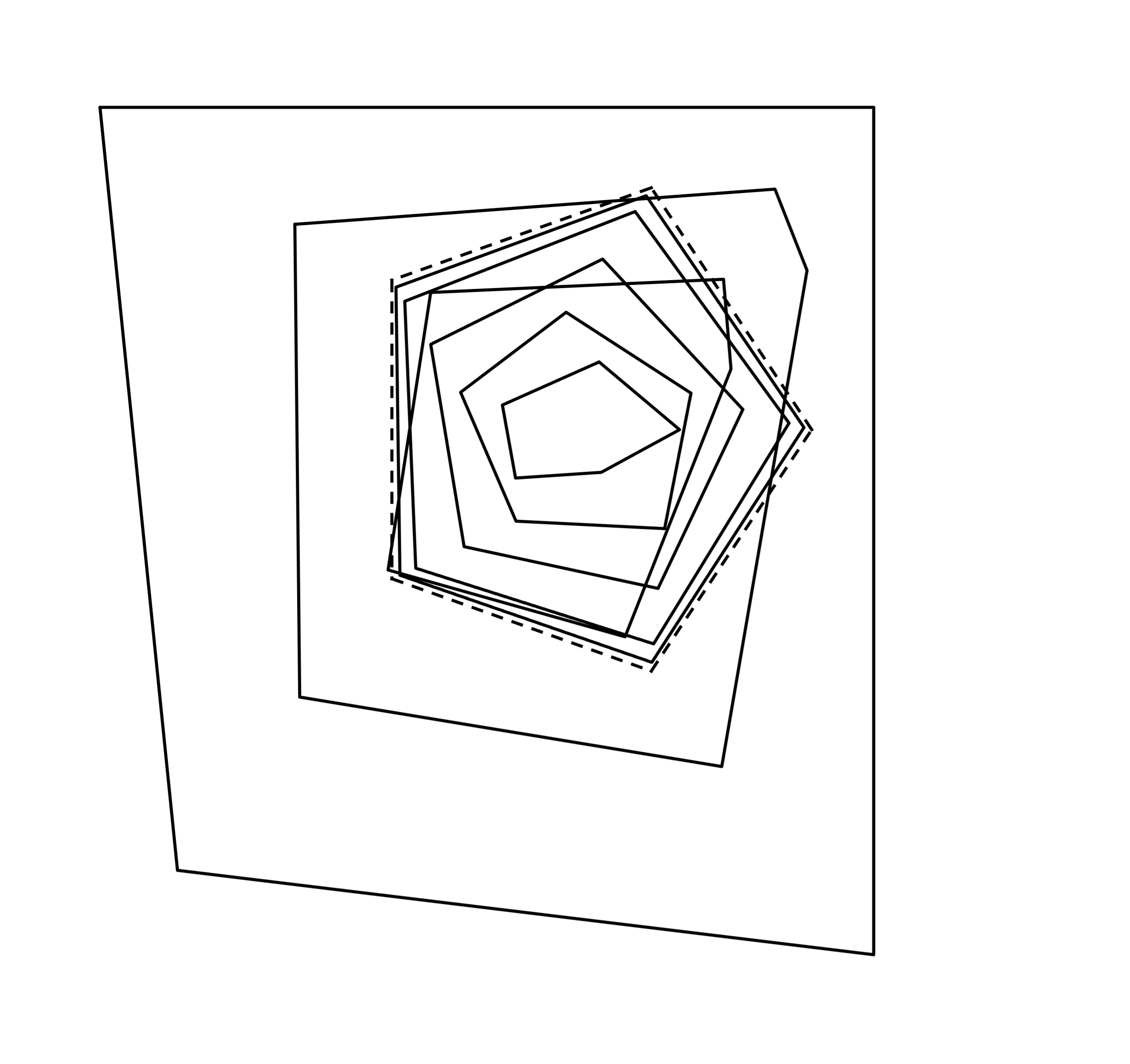}
    \caption{Initial polygon contains duplicated consecutive vertices to represent a quadrilateral flowing to a regular pentagon target by \eqref{eqn:Yau_diff_eqn} for $m=1.$ }
    \label{fig:Yau_multi_m=1}
  \end{subfigure}
  \hfill
  \begin{subfigure}[b]{0.31\textwidth}
    \includegraphics[width=\textwidth]{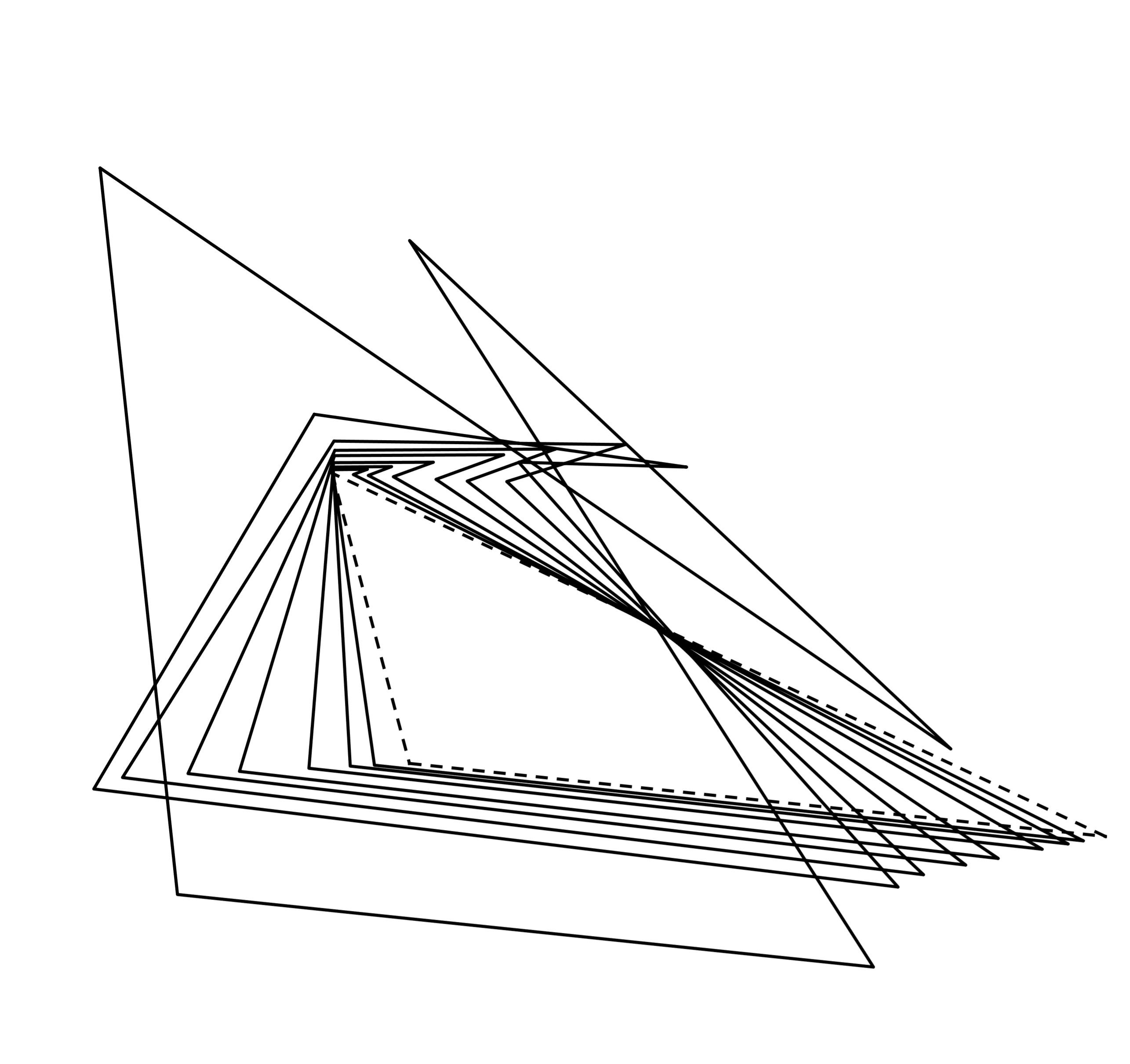}
    \caption{An initial pentagon flows by \eqref{eqn:Yau_diff_eqn} for $m=3$ to a target polygon that contains three duplicated consecutive vertices to represent a triangle.}
    \label{fig:Yau_multi_m=3}
  \end{subfigure}
  \hfill
   \begin{subfigure}[b]{0.31\textwidth}
    \includegraphics[width=\textwidth]{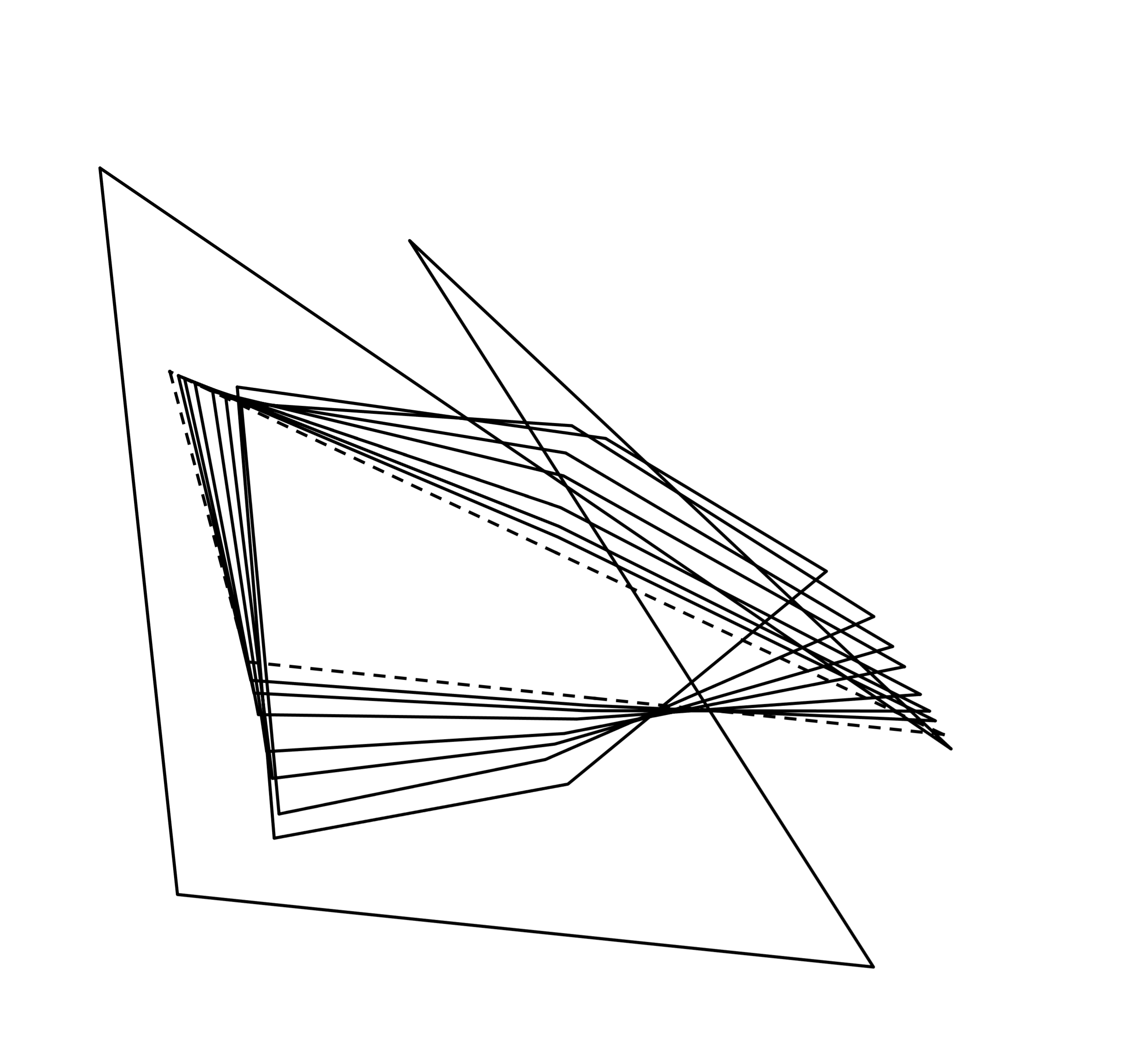}
    \caption{An initial pentagon flows by \eqref{eqn:Yau_diff_eqn} for $m=3$ to a target polygon that contains vertices at midpoints of line segments to represent a triangle.}
    \label{fig:Yau_mid_m=3}
  \end{subfigure}
 
  \caption{Examples of flowing from an initial polygon to a target polygon with a different number of vertices by the semi-discrete linear Yau difference flow. Selected time steps of evolution are shown superimposed over the initial and target polygons. The target polygon is given by the dashed line.} \label{fig:Yau_vertex_cases}
\end{figure*}

\begin{remark}
\begin{enumerate}
    \item Parabolic curvature flows for evolving one smooth curve to another were described in \cite{lin2009evolving, MSW23} and the references therein.  
    \item The above case for $m=1$ corresponds to the inhomogeneous version of Chow and Glickenstein's flow in \cite{chow2007semidiscrete}.
    \item The expression \eqref{eqn:homogenous_solution} provides an alternative, equivalent expression for the solution of the homogeneous semi-discrete polyharmonic flow \eqref{eqn:polyflow}.
    \item Instead of flowing to a fixed polygon $Y$ we could flow to $Y(t)$.  The difference above would be in the calculation of $X_p(t)$.  For example, if $Y(t) = \mu(t) Y_0$ for some scaling factor $\mu(t)$, this factor flows through to the integrand in the formula for $X_p(t)$, and an explicit solution can be written down for various functions $\mu(t)$.  Flowing to a scaling of $Y_0$ is reminiscent of Type I behaviour in various smooth curvature flows, where under appropriate rescaling singularities are modelled by self-similar scaling solutions.
    \item The evolution equation \eqref{eqn:Yau_diff_eqn} can flow \emph{any} initial polygon $X^0$ with $n$ sides to \emph{any} target polygon $Y$ with $n$ sides.  In contrast, flowing one smooth curve to another typically requires conditions on the initial curve at least (e.g. embeddedness, convexity).  Of course, as in the smooth case (e.g. in \cite{lin2009evolving}), there are also other ways of deforming one polygon to another, that do not involve a curvature flow.  In our setting, we could simply take, for example, $X\colon\left[ 0, 1\right] \rightarrow \mathbb{R}^p$ given by
$$X\left(t\right) = t Y + \left( 1-t\right) X^0 \mbox{.}$$
  \item As with our earlier flows, one may consider ancient solutions to \eqref{eqn:Yau_diff_eqn}.  Taking $t\rightarrow -\infty$, we see that under appropriate rescaling, solutions approach an affine transformation of the regular polygon with the least dominant eigenvalue. Again a difference with the smooth case is that any solution may be extended back in time.

\end{enumerate}
   \end{remark}

\end{document}